\documentclass[reqno]{amsart}
\usepackage{textcomp}
\usepackage{amsthm}
\usepackage{amsmath}
\usepackage{amssymb}
\usepackage{amsrefs}
\usepackage{enumerate}
\usepackage{hyperref}
\usepackage[capitalize]{cleveref}

\newtheorem{thm}{Theorem}[section]
\newtheorem{lem}[thm]{Lemma}

\newtheorem{cor}[thm]{Corollary}

\newcounter{casecounter-mainthm}
\newcounter{casecounter-fundamental-lemma}

\newcounter{introthmcounter}

\newtheorem{thmAlph}[introthmcounter]{Theorem}

\theoremstyle{definition}
\newtheorem{defn}[thm]{Definition}

\newtheorem{fact}[thm]{Fact}

\newtheorem{case-mainthm}[casecounter-mainthm]{Case}

\newtheorem{claim-defopen}{Claim}

\newtheorem{case-fundamental-lemma}[casecounter-fundamental-lemma]{Case}

\newcommand{\Sub}{\text{Sub}}
\newcommand{\tor}{\text{tor}}

\newcommand{\rc}{\text{rc}}

\newcommand{\GD}{\Gamma \Delta}
\newcommand{\RGD}{\text{Re}(\Gamma\Delta)}

\newcommand{\dcl}{\text{dcl}}
\newcommand{\tp}{\text{tp}}

\newcommand{\p}{\prime}
\newcommand{\pp}{\prime \prime}
\newcommand{\sptp}{\text{sptp}}

\newcommand{\lorm}{\mathcal{L}_{orm}}
\newcommand{\lom}{\mathcal{L}_{om}}
\newcommand{\id}{\text{id}}
\newcommand{\orm}{\text{orm}}

\newcommand{\intr}{\text{int}}

\newcommand{\bdelta}{\varepsilon}

\renewcommand{\bar}[1]{\overline{#1}}
\renewcommand{\Re}{\text{Re}}
\renewcommand{\Im}{\text{Im}}
\renewcommand{\lor}{\mathcal{L}_{or}}

\newcommand{\av}[1]{\left\lvert#1\right\rvert}
\newcommand{\ip}[1]{\langle#1\rangle}
\newcommand{\mbb}[1]{\mathbb{#1}}
\newcommand{\mc}[1]{\mathcal{#1}}

\newcommand{\rbar}{\bar{\mathbb{R}}}

\makeatletter
\let\xx@thm\@thm
\AtBeginDocument{\let\@thm\xx@thm}
\makeatother

\title{On expansions of the real field by complex subgroups}
\author{Erin Caulfield}
\address
{Department of Mathematics\\University of Illinois at Urbana-Champaign\\1409 West Green Street\\Urbana, IL 61801}
\email{ecaulfi2@illinois.edu}

\date{\today}

\begin{document}

\begin{abstract}
We construct a class of finite rank multiplicative subgroups of the complex numbers such that the expansion of the real field by such a group is model-theoretically well-behaved. As an application we show that a classification of expansions of the real field by cyclic multiplicative subgroups of the complex numbers due to Hieronymi does not even extend to expansions by subgroups with two generators.
\end{abstract}

\maketitle

\section{Introduction}

Let $\rbar := (\mbb{R},<,+,\cdot)$ be the real field. This paper contributes to the classification
of expansions of $\rbar$ by finite rank multiplicative subgroups $S$ of complex numbers. Here we identify $\mbb{C}$ with $\mbb{R}^2$ as usual and consider expansions of $\rbar$ by a binary predicate for the multiplicative subgroup. This is not the first time such structures have been studied. Belegradek and Zilber \cite{BelZilber} and independently G\"unayd\i n \cite{Ayhan} initiated the study of such expansions by fully determining the model theory of such expansions when $S$ is a finite rank subgroup of the unit circle $\mathbb{S}^1$. Using this work Hieronymi \cite{Philipp} established that if $S$ is a cyclic subgroup of $\mbb{C}$ (not necessarily a subgroup of $\mathbb{S}^{1}$), then  exactly one of the
following statements holds:
\begin{itemize}
\item[(i)] $(\rbar, S)$ defines $\mbb{Z}$,
\item[(ii)] $(\rbar, S)$ is d-minimal, or
\item[(iii)] every open definable set in $(\rbar, S)$ is semialgebraic.
\end{itemize}
An ordered structure $\mc{R}$ is \emph{d-minimal} if for every $M \equiv \mc{R}$, every subset of $M$ is a disjoint union of open intervals and finitely many discrete sets. More is known: By Theorem 1.3 in G\"unayd\i n and Hieronymi \cite{AyhanPhilipp} every finite rank subgroup of $\mbb{S}^1$ satisfies (iii), and therefore this classification can be extended to include such groups. This leads naturally to the question whether this holds true for arbitrary finite rank subgroups. The main result of this paper is a negative answer to this question.

\begin{thmAlph}\label{main-thm-statement}
Let $\Gamma$ be a finite rank subgroup of $\mbb{S}^1$ which is dense in $\mbb{S}^1$, let $\Delta = \bdelta^{\mbb{Z}}$ for $\bdelta \in \mbb{R}^{>0}$, and set $S = \GD$. Then every subset of $\mbb{R}^m$ definable in $(\rbar,S)$ is a Boolean combination of sets of the form
\[
\{x \in \mbb{R}^m : \exists y \in S^n \text{ s.t. } (x,y) \in W\}
\]
for some semialgebraic set $W \subseteq \mbb{R}^{m+2n}$. Moreover, every open definable set in $(\rbar,S)$ is definable in $(\rbar,\Delta)$.
\end{thmAlph}

It is not hard to see that $(\rbar,S)$ does not satisfy any of (i)-(iii). First note that $(\rbar,S)$ defines both $\Gamma$ and $\Delta$. If $(\rbar,S)$ defines $\mbb{Z}$, then by \cite[(37.6)]{Kechris}, $(\rbar,S)$ defines every projective subset of $\mbb{R}$. However, it can be checked that $(\rbar,S)$ does not define every projective subset of $\mbb{R}$. For example, if $S$ is countable, then every subset of $\mbb{R}$ which is definable in $(\rbar,S)$ is a Boolean combination of $F_{\sigma}$ sets by \autoref{main-thm-statement}. The projection of $S$ onto the real line is a definable set that is dense and codense, so $(\rbar,S)$ is not d-minimal. Lastly, the complement of $\Delta$ in the real line is open and definable in $(\rbar,S)$, but is not semialgebraic. By picking $\Gamma$ to be the group generated by $e^{i \pi \varphi}$ for some irrational $\varphi \in \mbb{R}^{>0}$, we see that the above classification fails even for multiplicative subgroups generated by two elements.\\

The fact that the sets definable in $(\rbar,S)$ have the form given in \cref{main-thm-statement} is proved in \cref{subsec:qr-thm}. We call this property quantifier reduction. The fact that every open definable set in $(\rbar,S)$ is definable in $(\rbar,\Delta)$ is proved in \cref{subsec:opendef}. In addition to \cref{main-thm-statement} we will also give an axiomatization for such structures in \cref{sec:elequiv}. Let $\Gamma$ be a dense subgroup of $\mbb{S}^1$, let $\Delta = \bdelta^{\mbb{Z}}$ for some $\bdelta \in \mbb{R}^{>0}$, and set $S=\Gamma \Delta$. Since both $\Gamma$ and $\Delta$ are definable in $(\rbar,S)$, we will consider the structure $(\rbar,\Gamma,\Delta)$ instead. We further expand this structure by constant symbols for each element in $\Re(\Gamma) \cup \Im(\Gamma)$ and $\Delta$.

\begin{thmAlph}\label{elequiv-statement}
Let $K$ be a real closed field. Let $G$ be a dense subgroup of $\mbb{S}^{1}(K)$ and let $\gamma \mapsto \gamma^{\p} : \Gamma \to G$ be a group homomorphism. For $\gamma \in \Gamma$ with $\gamma = (\alpha,\beta)$, let $\alpha^{\p}$ and $\beta^{\p}$ be such that $\gamma^{\p} = (\alpha^{\p},\beta^{\p})$. Let $A$ be a subgroup of $K^{>0}$ with a group homomorphism $\delta \mapsto \delta^{\prime} : \Delta \to A$ such that
\begin{itemize}
\item [(i)] $\bdelta^{\prime}$ is the smallest element of $A$ greater than 1, and
\item [(ii)] for every $k \in K^{>0}$, there is $a \in A$ such that $a \leq k < a\bdelta^{\prime}$.
\end{itemize}
 Then
\[
(K,G,A,(\delta^{\prime})_{\delta \in \Delta}, (\gamma^{\prime})_{\gamma \in \Gamma}) \equiv (\bar{\mbb{R}},\Gamma,\Delta,(\delta)_{\delta \in \Delta},(\gamma)_{\gamma \in \Gamma})
\]
if and only if:
\begin{enumerate}
\item for every $\gamma \in \Gamma$ and $n \in \mbb{Z}^{>0}$, $\gamma$ is an $n$th power in $\Gamma$ if and only if $\gamma^{\prime}$ is an $n$th power in $G$;
\item for all primes $p$, $[p]\Gamma = [p]G$;
\item for all $n \in \mbb{Z}^{>0}$, all polynomials $Q(x_1,\ldots,x_n) \in \mbb{Z}[x_1,\ldots,x_n]$, and all tuples $(\gamma_1\delta_1,\ldots,\gamma_n\delta_n)$ of elements of $\GD$,
\[
Q(\Re(\gamma_1\delta_1),\ldots,\Re(\gamma_n\delta_n)) > 0 \text{ if and only if } Q(\Re(\gamma_1^{\prime}\delta_1^{\prime}),\ldots,\Re(\gamma_n^{\prime}\delta_n^{\prime})) > 0;
\]
\item $(K,GA,(\gamma^{\prime}\delta^{\prime})_{\gamma \in \Gamma, \delta \in \Delta})$ satisfies the Mann axioms for $\GD$;
\item all torsion points of $G$ are in $\Gamma$.
\end{enumerate}
\end{thmAlph}

The notation $[p]\Gamma$ is defined in \cref{subsec:abelian-groups}. The statement of the Mann axioms is too technical for the introduction and we postpone it until \cref{subsec:mannaxioms}. A few comments about the methods used in this paper and their origins are in order. In \cite[Chapter 6]{Ayhan} an axiomatization for expansions $(\rbar,\Lambda,\Delta)$ of $\rbar$ by a dense multiplicative subgroup $\Lambda$ of $\mbb{R}^{>0}$ and a discrete multiplicative subgroup of $\Delta$ of $\mbb{R}^{>0}$ is presented. The methods used there can be adjusted to work when $\Lambda$ is replaced by a subgroup $\Gamma$ of $\mbb{S}^{1}$. This allows us to prove \cref{elequiv-statement} in \cref{sec:elequiv}. We then use a back-and-forth system constructed in the proof of \cref{elequiv-statement} to prove \cref{main-thm-statement}. For this second step, we can not follow \cite{Ayhan}, because no quantifier reduction statement along the lines of \cref{main-thm-statement} was established there. However, it is not hard to check that the method we use here can be transferred to structures considered in \cite{Ayhan} to give quantifier elimination statements comparable to \cref{main-thm-statement}.

\section{Prerequisites}

Throughout, let $V$ be a unary predicate and $P$ a binary predicate. Let $\lor := \{+,\cdot,-,0,1,<\}$ denote the language of ordered rings. Let $\lom := \{\cdot,1,<\}$ denote the language of ordered monoids. By ``definable'', we mean ``definable with parameters'' unless stated otherwise.

For a real closed field $K$, let $\mbb{S}^1(K) := \{(x,y) \in K^2 : x^2+y^2 = 1\}$. Throughout, we will identify $K(i)$ with $K^2$ in the same way that we identify $\mbb{C}$ with $\mbb{R}^2$. Let $\Re: K^2 \to K$ denote the usual ``real part'' function and let $\Im: K^2 \to K$ denote the usual ``imaginary part'' function. That is, $\Re$ is projection onto the first coordinate and $\Im$ is projection onto the second coordinate.

For a language $\mc{L}$ and set $S$, by $\mc{L}(S)$ we mean the language consisting of $\mc{L}$ together with constants for each element of the set $S$. If $S$ is a subset of $\mbb{C}$, then by $\mc{L}(S)$ we mean $\mc{L}(\Re(S),\Im(S))$ unless specified otherwise. We denote the real closure of a field $F$ by $F^{\rc}$. Let $K^{\p}$ be a subfield of a real closed field $K$ and $S \subseteq K(i)$. When we write $K^{\p}(S)^{\rc}$, we mean $K^{\p}(\Re(S),\Im(S))^{\rc}$.

Throughout, unless specified otherwise, $d$, $m$, and $n$ will range over positive natural numbers.

\subsection{Abelian groups}\label{subsec:abelian-groups}

Let $A$ be a multiplicatively written abelian group.
\begin{enumerate}
\item Let $A^{[n]} = \{a^n : a \in A\}$.
\item Let $A_{\tor} = \{a \in A : a^n = 1 \text{ for some } n \geq 1 \}$, the torsion subgroup of $A$.
\item Let $[d]A$ be equal to $\av{A/A^{[d]}}$ if $A/A^{[d]}$ is finite, and $\infty$ otherwise.
\end{enumerate}

Let $B$ be a subgroup of $A$. We say that $B$ is \emph{pure in} $A$, or that $B$ is a pure subgroup of $A$, if $B \cap A^{[m]} = B^{[m]}$ for all $m \geq 1$. That is, an element in $B$ has an $m$th root in $A$ if and only if it has an $m$th root in $B$.\\

Let $B \subseteq A$ be a subgroup and let $S \subseteq A$. We define $B\ip{S}_A$ to be the subgroup of $A$ given by
\[
\displaystyle B\ip{S}_A := \left\{ a \in A : a^n = a^{\prime}s_1^{k_1} \ldots s_m^{k_m}, a^{\prime} \in B, s_1,\ldots,s_m \in S, k_1,\ldots,k_m \in \mbb{Z}, \right.
\]
\[ \displaystyle\left. m \geq 0, n > 0 \right\}.
\]
If $S = \{a\}$ for some $a \in A \setminus B$, we write $B\ip{a}_A$ instead of $B\ip{S}_A$. Note that $B\ip{S}_A$ is a pure subgroup of $A$. Also note that
\[
B\ip{a_1,\ldots,a_n}_A = ( \ldots (B\ip{a_1}_A)\ip{a_2}_A) \ldots )\ip{a_n}_A
\]
for $a_1,\ldots,a_n \in A$.

\subsection{Fields}
Let $E$ and $F$ be field extensions of a field $k$, where $E,F$ are subfields of a field $K$. We say that $E$ and $F$ are \emph{free over $k$} if any set $S \subseteq E$ which is algebraically independent over $k$ is also algebraically independent over $F$. Equivalently, $E$ and $F$ are free over $k$ if any $S \subseteq E$ which is algebraically dependent over $F$ is algebraically dependent over $k$.

The next fact follows from Proposition 12 in Section 14, Chapter V of \cite{Bourbaki}.

\begin{fact}\label{prop12-bourbaki}
Let $E$ and $F$ be field extensions of a field $k$, where $E,F$ are subfields of a field $K$. Then $E$ and $F$ are free over $k$ if and only if there exists a transcendence basis of $E$ over $k$ which is algebraically independent over $F$.
\end{fact}

The next fact follows easily from the definition of freeness. This fact is also part of Exercise 14 in Section 14, Chapter V in \cite{Bourbaki}.

\begin{fact}\label{exercise14-bourbaki}
Let $E,F,G$ be three extensions of a field $k$ contained in a field $K$ such that $F \subseteq G$. If $E$ and $F$ are free over $k$ and $E(F)$ and $G$ are free over $F$, then $E$ and $G$ are free over $k$.
\end{fact}

%

\subsection{The Archimedean valuation}
Let $K$ be an ordered field and let $K^{\times}$ denote the nonzero elements of $K$.
For $k \in K$, let $\av{k} := \max\{k, -k\}$. We define an equivalence relation $\sim$ on $K^{\times}$ by setting $x \sim y$ (for $x,y \in K^{\times}$) if and only if there is $n \geq 1$ ($n \in \mbb{N}$) such that
\[
\displaystyle\frac{1}{n} < \av{\frac{y}{x}} < n.
\]
Here $n$ denotes $1+1+\ldots+1$, $n$ times, and $1/n$ denotes its inverse. We say that a positive element $k \in K$ is \emph{finite} if there is some $n \in \mbb{N}$ such that $k < n$.\\

It it is easy to see that $\sim$ is an equivalence relation. Denote the natural quotient map from $K^{\times}$ to $K^{\times}/{\sim}$ by $v$.  For $x,y \in K^{\times}$, we define an operation $+: v(K^{\times}) \times v(K^{\times}) \to v(K^{\times})$ by $v(x)+v(y) := v(xy)$. For $x,y \in K^{\times}$, we define an ordering on $v(K^{\times})$ by taking $v(x) > v(y)$ if and only if $\av{\frac{x}{y}} < \frac{1}{n}$ for all $n \geq 1$. It can be shown that these definitions do not depend on our choice of representative, and ($K^{\times}/{\sim}$,$+$,$<$) is an ordered abelian group. With these definitions, the map $v$ is a valuation on $K$. We call $v$ the \emph{Archimedean valuation} on $K$.\\

Let
\[
R = \{x \in K : \av{x} < n \text{ for some } n \geq 1\}.
\]
Note that $R$ is a convex subring of $K$, and the maximal ideal of $R$ is the set $\{x \in K : \av{x} < \frac{1}{n} \text{ for all } n \geq 1\}$. Thus, $R$ is a valuation ring of $K$, and the Archimedean valuation $v$ is the associated valuation on $K$.\\

If $K$ is a real closed field, then we consider $v(K^{\times})$ as a $\mbb{Q}$-vector space as follows. We define scalar multiplication on $v(K^{\times})$ by $q \cdot v(a) = v(\av{a}^q)$ for $a \in K^{\times}$ and $q \in \mbb{Q}$. It can be shown that these addition and scalar multiplication operations are well-defined by definition of $v$. Moreover, these operations make $v(K^{\times})$ into a $\mbb{Q}$-vector space. Throughout, when we refer to a valuation $v$ on an ordered field $K$, $v$ is the Archimedean valuation on $K$. The following theorem is Corollary 5.6 in van den Dries \cite{TconvexII}, translated to fit our situation. We call this theorem the valuation inequality.

\begin{thm}[Valuation inequality] Let $K$ be a real closed field and let $K^{\p}$ be a real closed subfield of $K$. Let $a \in K$, and let $V = v(K^{\p}(a)^{\rc})$, $W = v(K^{\p})$. As a $\mbb{Q}$-vector space, we have $\dim(V/W) \leq 1$.\end{thm}

\subsection{Regularly discrete groups} If $A$ is an ordered abelian group, we say that $S \subseteq A$ is \emph{convex} if for all $a,b \in S$ and all $x \in A$ such that $a < x < b$, we have $x \in S$.

\begin{defn}
An ordered abelian group $A$ with a smallest element larger than 1 is said to be \emph{regularly discrete} if for all $n \geq 1$ and every infinite convex set $S \subseteq A$, $S \cap A^{[n]} \neq \emptyset$.
\end{defn}
It can be shown that if $A$ is a multiplicative subgroup of a real closed field $K$ with a smallest element larger than 1, then $A$ is regularly discrete. The following lemma follows from the proof of Theorem 2.1 in Zakon \cite{Zakon}.

\begin{lem}\label{reg-discrete-lem} Let $A$ be an ordered abelian group with a smallest element larger than 1, denoted $\bdelta$. The following are equivalent:
\begin{enumerate}
\item $A$ is regularly discrete;
\item for all $n \geq 1$, $A/A^{[n]} = \{\bdelta A^{[n]}, \bdelta^2 A^{[n]}, \ldots, \bdelta^n A^{[n]}\}$.
\end{enumerate}
\end{lem}
In particular, an ordered abelian group $A$ with a smallest element larger than 1 is regularly discrete if and only if $[n]A = n$ for all $n \geq 1$.

It is easy to see that $A$ is regularly discrete if and only if for all $n \geq 1$ and all $a,b \in A$ such that $(a,b)$ has at least $n$ elements, $A^{[n]} \cap (a,b) \neq \emptyset$.

\subsection{Discrete multiplicative subgroups}
Let $K$ be a real closed field and $A \subseteq K^{>0}$ be a subgroup of $K^{\times}$ with a smallest element larger than 1, $\bdelta$. Suppose that for every $k \in K^{>0}$, there is $a \in A$ such that $a \leq k < a\bdelta$.
Define a function $\lambda: K \to A$ by
\[
\lambda(k) = \begin{cases}
0, & k \leq 0 \\
a, & k > 0 \text{ and } a \leq k < a\bdelta
\end{cases}
\]
It is easy to see that for $k \in K^{>0}$, there is exactly one element $a \in A$ such that $a \leq k < a\bdelta$. Therefore, $\lambda$ is well-defined.

\begin{lem}\label{lambda-field-ops}
For $k_1,k_2 \in K^{>0}$, $\lambda(k_1k_2) = \lambda(k_1)\lambda(k_2)$ or $\lambda(k_1k_2) = \lambda(k_1)\lambda(k_2)\bdelta$. In general, for $k_1,\ldots,k_n \in K^{>0}$, $\lambda(k_1 \ldots k_n) = \lambda(k_1)\ldots \lambda(k_n)\bdelta^j$ for some $j \in \{0,\ldots,n-1\}$.
\end{lem}

\begin{proof}
Let $a_1 = \lambda(k_1)$, $a_2 = \lambda(k_2)$. Note that in particular, $a_1,a_2 > 0$. We have $a_1a_2 \leq k_1k_2 < a_1\bdelta a_2\bdelta = a_1a_2\bdelta^2$. Moreover, $a_1a_2 < a_1a_2\bdelta < a_1a_2\bdelta^2$, so by definition of $\lambda$, we must have $\lambda(k_1k_2) = a_1a_2$ or $\lambda(k_1k_2) = a_1a_2\bdelta$. The proof of the last statement is similar.
\end{proof}

\subsection{Oriented abelian groups}\label{subsec:orag} Oriented groups were introduced by G\"unayd\i n in \cite{Ayhan}. We refer the reader to Section 8.1 of \cite{Ayhan} for the precise definition of an oriented group $G$ with orientation $\mc{O}$.

Let $\lorm$ be the language of oriented monoids; that is, $\lorm = \{\mc{O},1,\cdot\}$, where $\mc{O}$ is a ternary relation.
Let $K$ be a real closed field and $G$ a multiplicative subgroup of $\mbb{S}^1(K) \subseteq K^2$. Thus, the identity of $G$ is the element $(1,0)$ of $K^2$ and multiplication on $G$ is defined by
\[
(x_1,y_1) \cdot (x_2,y_2) = (x_1x_2-y_1y_2, x_1y_2 + y_1x_2).
\]
where multiplication and addition in the components on the right side are performed in $K$. We can make $G$ into an oriented subgroup by taking the orientation $\mc{O}$ on $G$ to be the one inherited from $\mbb{S}^1(K)$. Thus, $\mc{O}$ is defined by the quantifier free $\lor$-formula $\varphi(x_1,y_1,x_2,y_2,x_3,y_3)$ discussed in the example in Section 8.1 of \cite{Ayhan}. We say that $G$ is \emph{dense} in $\mbb{S}^1(K)$ if for all $a,b \in \mbb{S}^1(K)$ such that $\mc{O}(1,a,b)$ holds, there is $g \in G$ with $\mc{O}(a,g,b)$. Note that such a $G$ is also \emph{regularly dense} in $\mbb{S}^1(K)$, that is for all $a,b \in \mbb{S}^1(K)$ such that $\mc{O}(1,a,b)$ holds and for all $n \geq 1$, there is $g \in G$ with $\mc{O}(a,g^n,b)$.\\


Let $\Gamma \subseteq G$ be an infinite subgroup. Let $x = (x_1,\ldots,x_n)$ be a tuple of $n$ variables, and let $z = ((z_{11},z_{12}),\ldots,(z_{n1},z_{n2}))$ be a tuple of $n$ pairs of variables. Let $\phi(x)$ be an $\lorm(\Gamma)$-formula. From the definition of multiplication, orientation, and identity in $G$, we see that there is an $\lor(\Gamma)$-formula $\psi_{\phi}(z)$ such that for all $(a_1,\ldots,a_n) \in G^n$ (with $a_i = (a_{i1},a_{i2})$ for some $a_{i1},a_{i2} \in K$),
\[
(G,\mc{O},1,\cdot) \models \phi(a_1,\ldots,a_n) \text{ if and only if } 
\]
\[
(K,<,+,-,0,1,\cdot) \models \psi_{\phi}((a_{11},a_{12}),\ldots,(a_{n1},a_{n2})).
\]
In particular, all quantifiers that appear in $\psi_{\phi}$ must appear in pairs of $\exists$ or $\forall$, and there must be an even number of free variables in $\psi_{\phi}$.\\

Let $\Sigma_{\orm}(\Gamma) := \{\psi_{\phi} : \phi \text{ an } \lorm(\Gamma) \text{-formula}\}$. Note that $\Sigma_{\orm}(\Gamma)$ is closed under conjunctions and disjunctions of formulas, negation, and quantification over a pair of variables.

\begin{defn}\label{defn-p-restriction}
Let $P$ be a binary predicate. The $P$-restriction $\theta_P$ of $\theta \in \Sigma_{\orm}(\Gamma)$ is defined recursively in analogy with the $U$-restriction defined on page 10 of \cite{Ayhan}. In the following, $\theta, \theta^{\p}$, and $\theta^{\pp}$ are formulas in $\Sigma_{\orm}(\Gamma)$.
\begin{itemize}
\item If $\theta$ is an atomic $\lor(\Gamma)$-formula, then $\theta_P := \theta$;
\item if $\theta = \neg \theta^{\p}$, then $\theta_P := \neg \theta_P^{\p}$;
\item if $\theta = \theta^{\p} \wedge \theta^{\pp}$, then $\theta_P := \theta_P^{\p} \wedge \theta_P^{\pp}$;
\item if $\theta = \theta^{\p} \vee \theta^{\pp}$, then $\theta_P := \theta_P^{\p} \vee \theta_P^{\pp}$;
\item if $\theta = \exists x_1 \exists y_1 \theta^{\p}$, then $\theta_P := \exists x_1 \exists y_1 (P(x_1,y_1) \wedge \theta_P^{\p})$;
\item if $\theta = \forall x_1 \forall y_1 \theta^{\p}$, then $\theta_P := \forall x_1 \forall y_1 (P(x_1,y_1) \rightarrow \theta_P^{\p})$.
\end{itemize}
\end{defn}

The following lemma is proved in Section 8.1.2 of \cite{Ayhan}. We will use this lemma in the proof of \cref{elequiv-statement} and to prove quantifier elimination in \cref{subsec:qe-lemmas}.

\begin{lem}\label{claim-812}
Let $A$ and $B$ be regularly dense oriented abelian groups such that $[p]A = [p]B$ for every prime $p$. Let $A^{\p}$ and $B^{\p}$ be pure subgroups of $A$ and $B$ respectively such that $A_{\tor}^{\p} = A_{\tor}$ and $B_{\tor}^{\p} = B_{\tor}$. Let $f: A^{\p} \to B^{\p}$ be an oriented abelian group isomorphism. Suppose that $B$ is $\kappa$-saturated, where $\kappa > \av{B^{\p}}$ is uncountable, and let $a \in A \setminus A^{\p}$. There is $b \in B$ such that there is an oriented group isomorphism $A^{\p}\ip{a}_A \to B^{\p}\ip{b}_B$ extending $f$ which takes $a$ to $b$.
\end{lem}

\section{An Axiomatization}

For the rest of this paper, we will fix a finite rank multiplicative subgroup $\Gamma$ of $\mbb{S}^1$ and a multiplicative subgroup $\Delta$ of $\mbb{R}^{>0}$ of the form $\bdelta^{\mbb{Z}}$ for some $\bdelta > 1$ in $\mbb{R}$.

\begin{defn} Let $K$ be a field and $G$ a multiplicative subgroup of $K$. For $a_1,\ldots,a_n \in \mbb{Q}$ ($n \geq 1$), a nondegenerate solution in $G$ to the equation
\begin{equation}\label{mann-eqn}
a_1x_1+\ldots+a_nx_n = 1
\end{equation}
is a tuple $(g_1,\ldots,g_n) \in G^n$ such that
\[
a_1g_1+\ldots+a_ng_n = 1
\]
and $\sum_{i \in I} a_ig_i \neq 0$ for each nonempty subset $I \subseteq \{1,\ldots,n\}$. We say that $G$ has the \emph{Mann property} if every equation of the form \cref{mann-eqn} has only finitely many nondegenerate solutions in $G$.
\end{defn}

Every finite rank multiplicative subgroup of a field of characteristic 0 has the Mann property. This follows directly from Theorem 1.1 of Evertse, Schlickewei, and Schmidt \cite{ESS}. In particular, since $\GD$ is a finite rank multiplicative subgroup of $\mbb{C}$, $\GD$ has the Mann property. Note that by \cite[Proposition 1.1]{BelZilber}, $[n]\Gamma$ is finite for each $n \geq 1$.

In this section, we give an axiomatization for a theory $T$ and deduce some first consequences of this theory.

\subsection{Orientation axioms and Mann axioms}\label{subsec:mannaxioms} We now define two important sets of axioms: orientation axioms and Mann axioms.

Given any polynomial $Q(x_1,\ldots,x_n) \in \mbb{Z}[x_1,\ldots,x_n]$ and tuple $\gamma\delta := (\gamma_1\delta_1,\ldots,\gamma_n\delta_n)$ of elements of $\GD$, we say the \emph{ordering axiom for $\gamma\delta$ and $Q$} is the sentence
\[
Q(\Re(\gamma_1\delta_1),\ldots,\Re(\gamma_n\delta_n)) > 0
\]
if this holds in $\mbb{R}$, and otherwise it is the sentence
\[
Q(\Re(\gamma_1\delta_1),\ldots,\Re(\gamma_n\delta_n)) \leq 0.
\]
The \emph{orientation axioms of $\GD$ } is the collection of these sentences for each $n$, each polynomial $Q \in \mbb{Z}[x_1,\ldots,x_n]$ and tuple $\gamma\delta$.\\

For every linear equation
\begin{equation*}
a_1x_1+\ldots+a_nx_n = 1, (n \geq 1, a_1,\ldots,a_n \in \mathbb{Q}^{\times})
\end{equation*}
take the finite list of its nondegenerate solutions in $\Gamma\Delta$, say
\[
\gamma_1\delta_1 = (\gamma_{11}\delta_{11},\ldots,\gamma_{1n}\delta_{1n}), \ldots, \gamma_k\delta_k = (\gamma_{k1}\delta_{k1},\ldots,\gamma_{kn}\delta_{kn}).
\]
For $n$-tuples $y = (y_1,\ldots,y_n)$ and $z = (z_1,\ldots,z_n)$ and $\delta \in \Delta$ let
$P(\delta^{-1}y,\delta^{-1}z)$ abbreviate
\[
P(\delta^{-1}y_1,\delta^{-1}z_1) \wedge \ldots \wedge P(\delta^{-1}y_n,\delta^{-1}z_n)
\]
and let $(y,z) = \gamma_j\delta_j$ abbreviate
\begin{align*}
y_1 = \Re(\gamma_{j1}\delta_{j1}) \wedge \ldots \wedge y_n = \Re(\gamma_{jn}\delta_{jn}) \\
\wedge z_1 = \Im(\gamma_{j1}\delta_{j1}) \wedge \ldots \wedge z_n = \Im(\gamma_{jn}\delta_{jn}).
\end{align*}

Let \emph{the Mann axiom of $\GD$ corresponding to the equation $a_1x_1+\ldots+a_nx_n = 1$} be the sentence
\begin{multline*}
\displaystyle \forall y \forall z \left[ \forall b \left( V(b) \wedge P(by,bz) \wedge \sum_{i=1}^n a_iy_i = 1 \wedge \sum_{i=1}^n a_iz_i = 0 \wedge \right. \right. \\
\left. \left. \bigwedge_I (\sum_{i \in I} a_iy_i \neq 0 \vee \sum_{i \in I} a_iz_i \neq 0) \right) \rightarrow \bigvee_{j=1}^k (y,z) = \gamma_j\delta_j \right].
\end{multline*}
Here the conjunction $\bigwedge_I$ is taken over all nonempty proper subsets $I$ of $\{1,\ldots,n\}$.\\

Let $K$ be a real closed field. Suppose the predicate $P$ is interpreted as a subgroup $G \subseteq \mbb{S}^1(K)$ and the predicate $V$ is interpreted as a subgroup $A \subseteq K^{>0}$. In this setting, the Mann axiom of $\GD$ corresponding to the equation
\[
a_1x_1+\ldots+a_nx_n = 1
\]
can be interpreted as follows. Let $\gamma_1\delta_1,\ldots,\gamma_n\delta_n$ be the solutions to this equation in $\GD$. Suppose we have $(y_1,z_1),\ldots,(y_n,z_n) \in GA$ such that
\[
a_1(y_1+iz_1)+\ldots+a_n(y_n+iz_n) = 1.
\]
In particular, we must have $\sum_{i=1}^n a_iz_i = 0$ and $\sum_{i=1}^n a_iy_i = 1$. Suppose also that for all proper subsets $I \subseteq \{1,\ldots,n\}$, we have $\sum_{i \in I} a_iy_i \neq 0$ and $\sum_{i \in I} a_iz_i \neq 0$. That is, $((y_1,z_1),\ldots,(y_n,z_n))$ is a nondegenerate solution of \cref{mann-eqn}. Then letting $y = (y_1,\ldots,y_n)$ and $z = (z_1,\ldots,z_n)$, we must have $(y,z) = \gamma_j\delta_j$ in the sense defined in above.

\subsection{The theory $T$} We now state the axioms for our theory.

Let $\lor(P,V,\Gamma,\Delta)$ denote the language consisting of the language of ordered rings together with a binary predicate symbol $P$, unary predicate symbol $V$, and constants for each element of $\Re(\Gamma) \cup \Im(\Gamma) \cup \Delta$. For $\gamma \in \Gamma$ with $\gamma = (\alpha,\beta)$, let $\gamma^{\p} := (\alpha^{\p},\beta^{\p})$. Although $\lor(P,V,\Gamma,\Delta)$-structures have constants for each element of $\Re(\Gamma) \cup \Im(\Gamma) \cup \Delta$, we will write $\lor(P,V,\Gamma,\Delta)$-structures in the form
\[
(K,G,A,(\gamma^{\p})_{\gamma \in \Gamma}, (\delta^{\p})_{\delta \in \Delta})
\]
for convenience.

\begin{defn}\label{defn-t} Let $T$ be the $\lor(P,V,\Gamma,\Delta)$-theory whose models have the form
\[
(K,G,A, (\gamma^{\prime})_{\gamma \in \Gamma}, (\delta^{\prime})_{\delta \in \Delta})
\]
such that:
\begin{enumerate}
\item $K$ is a real closed field
\item $A$ is a subgroup of $K^{>0}$ such that $\bdelta^{\prime}$ is the smallest element of $A$ larger than 1
\item $G$ is a dense subgroup of $\mathbb{S}^1(K)$
\item for all $k \in K^{>0}$, there is $a \in A$ such that $a \leq k < a\bdelta$
\item $\delta \mapsto \delta^{\prime}: \Delta \to A$ and $\gamma \mapsto \gamma^{\p}$: $\Gamma \to G$ are group homomorphisms
\item $(K,(\gamma^{\prime}\delta^{\prime})_{\gamma \in \Gamma,\delta \in \Delta})$ satisfies the orientation axioms for $\Gamma\Delta$.
\item $(K,GA,(\gamma^{\prime}\delta^{\prime})_{\gamma \in \Gamma, \delta \in \Delta})$ satisfies the Mann axioms for $\GD$, $P$
\item $G_{\tor} = \Gamma_{\tor}$
\end{enumerate}
\end{defn}

Since $A$ is a multiplicative subgroup of the real closed field $K$ with a smallest element larger than 1, in particular, $A$ is regularly discrete. By \cref{reg-discrete-lem}, for each $n > 0$ and each $a \in A$, there is $i \in \{1,\ldots,n\}$ such that $a$ is congruent to $(\bdelta^{\p})^i$ modulo $A^{[n]}$.


For convenience, we will identify the subgroup $\Gamma^{\prime}$ of $K^2$ with $\Gamma$ and the subgroup $\Delta^{\prime}$ of $K$ with $\Delta$. Thus, we will write $\gamma$ rather than $\gamma^{\prime}$ and $\delta$ rather than $\delta^{\prime}$.

\subsection{Consequences of $T$.} Throughout this section, let $\mc{M} := (K,G,A,(\gamma)_{\gamma \in \Gamma},(\delta)_{\delta \in \Delta})$ be a model of $T$. In this subsection we derive some consequences of $T$. We begin by noting that since $(K,(\gamma\delta)_{\gamma \in \Gamma,\delta \in \Delta})$ satisfies the orientation axioms for $\GD$, $\bdelta$ is finite.

\begin{lem}\label{lambda-closure-lem1}
For all finite $x \in K^{>0}$, there is $l \in \mbb{Z}$ such that $\bdelta^l \leq x < \bdelta^{l+1}$. Therefore, for each finite $x \in K^{>0}$, $\lambda(x) = \bdelta^l$ for some $l \in \mbb{Z}$.
\end{lem}

\begin{proof}
First note that $\bdelta^{\mbb{Z}}$ is cofinal in the set of finite elements of $K$. This is because $\bdelta^{\mbb{Z}}$ is cofinal in $\mbb{N} \subseteq \mbb{R}$, and $(K,(\gamma\delta)_{\gamma \in \Gamma,\delta \in \Delta})$ satisfies the orientation axioms for $\GD$ by assumption. If $y \in K^{>0}$ is finite and $y > 1$, let $l$ be the smallest natural number such that $y < \bdelta^{l+1}$. If $0 < y \leq 1$, let $m$ be the smallest natural number such that $y < \bdelta^{-m+1}$, and then take $l = -m$.
\end{proof}

\begin{lem}\label{lambda-closure-lem3}
Let $K^{\p} = \mbb{Q}(\RGD)^{\rc}$. Then $\lambda((K^{\p})^{>0}) = \Delta$.
\end{lem}
\begin{proof}
It follows from the orientation axioms that every positive element of $K^{\p}$ is finite. By \cref{lambda-closure-lem1}, $\lambda((K^{\p})^{>0}) = \Delta$.
\end{proof}

\begin{lem}\label{lambda-closure-lem2} $\Delta$ is a pure subgroup of $A$.
\end{lem}

\begin{proof}
Let $a \in A$ be such that $a^n \in \Delta$ for some $n > 0$. Since $a^n \in \Delta$, there is $l \in \mbb{Z}$ such that $a^n = \bdelta^l$. If $a \geq 1$, then since $A \subseteq K^{>0}$, we have $0 < a \leq a^n = \bdelta^l$. In particular, $a$ is finite. If $0 < a < 1$, then clearly $a$ is finite.

Suppose for a contradiction that $a \notin \Delta$. Then since $a$ is finite, by \cref{lambda-closure-lem1}, there is $k \in \mbb{Z}$ such that $\bdelta^k < a < \bdelta^{k+1}$. But then $1 < a\bdelta^{-k} < \bdelta$, contradicting our assumption that $\bdelta$ is the smallest element of $A$ larger than 1. So we must have $a \in \Delta$.
\end{proof}

We will use the following lemma repeatedly in our arguments.

\begin{lem}[Fundamental Lemma]\label{fundamental-lemma} Let $K^{\p}$ be a subfield of $K$ which contains $\Delta$ and is closed under $\lambda$. Let $a \in K$. Then either $\lambda(K^{\p}(a)^{\rc}) = \lambda(K^{\p})$ or there is an $\lor(K^{\p})$-definable function $f$ such that $\lambda((K^{\p}(a)^{\rc})^{>0}) = \lambda((K^{\p})^{>0})\ip{\lambda(f(a))}_A$. If $a \in A$, then we may take $f = \id$ in the second case.\end{lem}

\begin{proof}
Note that since $K^{\p}$ is closed under $\lambda$, we have $\lambda((K^{\p})^{>0}) = K^{\p} \cap A$. Let $A^{\p} = K^{\p} \cap A$. We will consider two cases.

\begin{case-fundamental-lemma}\label{case1-fundamental-lemma} For all $x \in (K^{\p}(a)^{\rc})^{\times}$, $v(x) \in v(K^{\p})$. \end{case-fundamental-lemma}

We will show that $\lambda(K^{\p}(a)^{\rc}) = \lambda(K^{\p})$. In this case, let $x \in (K^{\p}(a)^{\rc})^{>0}$. By assumption, there is $y \in K^{\p}$ such that $v(x) = v(y)$. Note that by definition of $\lambda$, for any $k \in K^{>0}$, $\lambda(k) \leq k < \lambda(k)\bdelta$; therefore, $v(x) = v(\lambda(x))$ and $v(y) = v(\lambda(y))$. Since we assume that $v(x) = v(y)$, $\frac{\lambda(x)}{\lambda(y)}$ is finite. Therefore, we must actually have $\frac{\lambda(x)}{\lambda(y)} = \bdelta^l$ for some $l \in \mbb{Z}$ by \cref{lambda-closure-lem1}. Since $\lambda(x), \lambda(y) \in A$, $\lambda(\lambda(x)\lambda(y)^{-1}) = \lambda(x)\lambda(y)^{-1}$. By our assumption that $K^{\p}$ is closed under $\lambda$ and contains $\Delta$, we have $\lambda(y)\bdelta^l \in A^{\p}$. Therefore, $\lambda(x) \in A^{\p}$. This proves that $\lambda(K^{\p}(a)^{\rc})) \subseteq \lambda(K^{\p})$. The other inclusion is clear, so in this case, we have $\lambda(K^{\p}(a)^{\rc}) = \lambda(K^{\p})$.

Note that we have actually proved that if $v((K^{\p}(a)^{\rc})^{\times}) = v((K^{\p})^{\times})$, then $\lambda(K^{\p}(a)^{\rc}) = \lambda(K^{\p})$.

\begin{case-fundamental-lemma}\label{case2-fundamental-lemma} There is $z \in K^{\p}(a)^{\rc}$ such that $v(z) \notin v(K^{\p})$. \end{case-fundamental-lemma}

We want to show that $\lambda((K^{\p})^{>0})\ip{\lambda(f(a))}_A = \lambda((K^{\p}(a)^{\rc})^{>0})$. Since $z \in K^{\p}(a)^{\rc}$, there is an $\lor(K^{\p})$-definable function $f$ such that $z = f(a)$. Let $x \in (K^{\p}(a)^{\rc})^{>0}$. By the valuation inequality, $v(x) = v(k)+q(v(z))$ for some $k \in K^{\p}$, $q \in \mbb{Q}$. So there is $N \in \mbb{N}$ such that $\frac{1}{N} < \frac{x}{kz^q} < N$. Since $v(k) = v(\lambda(k))$ and since $K^{\p}$ is closed under $\lambda$, we have $\frac{1}{N} < \frac{x}{\lambda(k)z^q} < N$. Let $q = p_1/p_2$, where $p_2 > 0$ and $p_1,p_2 \in \mbb{Z}$. Then $\frac{1}{N^{p_2}} < \frac{x^{p_2}}{\lambda(k)^{p_2}z^{p_1}} < N^{p_2}$, so there is $M \in \mbb{N}$ such that $\frac{1}{M} < \frac{\lambda(x)^{p_2}}{\lambda(k)^{p_2}\lambda(z)^{p_1}} < M$. Since $\frac{\lambda(x)^{p_2}}{\lambda(k)^{p_2}\lambda(z)^{p_1}} \in A$ and this element is finite, there is $l \in \mbb{Z}$ such that $\frac{\lambda(x)^{p_2}}{\lambda(k)^{p_2}\lambda(z)^{p_1}} = \bdelta^l$. By definition of $A^{\p}\ip{\lambda(z)}_A$, we have $\lambda(x) \in A^{\p}\ip{\lambda(z)}_A$. Therefore, $\lambda(x) \in A^{\p}\ip{\lambda(f(a))}_A = \lambda((K^{\p})^{>0})\ip{\lambda(f(a))}_A$.

Now let $x \in \lambda((K^{\p})^{>0})\ip{\lambda(f(a))}_A$. We want to show that $x \in \lambda((K^{\p}(a)^{\rc})^{>0})$. By definition of $\lambda((K^{\p})^{>0})\ip{\lambda(f(a))}_A$ and our assumption that $\lambda((K^{\p})^{>0}) = A^{\p}$, there are $a^{\p} \in A^{\p}$, $d > 0$, and $l \in \mbb{Z}$ such that $x = (a^{\p}\lambda(f(a))^l)^{1/d}$. Since $v(f(a)) = v(\lambda(f(a)))$, we have $v(f(a)^l) = v(\lambda(f(a))^l)$. Therefore, there is $N_1 \in \mbb{N}$ such that $\frac{1}{N_1} < \frac{(a^{\p}f(a)^l)^{1/d}}{(a^{\p}\lambda(f(a))^l)^{1/d}} < N_1$. That is, $\frac{1}{N_1} < \frac{(a^{\p}f(a)^l)^{1/d}}{x} < N_1$. Moreover, $v((a^{\p} f(a)^l)^{1/d}) = v(\lambda((a^{\p}f(a)^l)^{1/d}))$, so there is $N_2 \in \mbb{N}$ such that $\frac{1}{N_2} < \frac{\lambda((a^{\p}f(a)^l)^{1/d})}{x} < N_2$. By \cref{lambda-closure-lem1}, there is $j \in \mbb{Z}$ such that $x = \bdelta^j \lambda((a^{\p}f(a)^l)^{1/d})$. Since $\lambda(\bdelta^j) = \bdelta^j$, by \cref{lambda-field-ops}, we have $x \in \lambda((K^{\p}(a)^{\rc})^{>0})$. Therefore, in this case, we have $\lambda((K^{\p})^{>0})\ip{\lambda(f(a))}_A = \lambda((K^{\p}(a)^{\rc})^{>0})$.\\

We now prove the last statement in the lemma. Suppose $a \in A$. We have two cases: $v(a) \in v(K^{\p})$ and $v(a) \notin v(K^{\p})$. In the first case, there is $k \in K^{\p}$ such that $v(k) = v(a)$. We may assume that $k > 0$ in $K^{\p}$. Since $v(k) = v(a)$ and $v(k) = v(\lambda(k))$, there is $n \in \mbb{N}$ such that $\frac{1}{n} < \frac{a}{\lambda(k)} < n$. Since $\frac{a}{\lambda(k)}$ is finite, there is $l \in \mbb{Z}$ such that $\frac{a}{\lambda(k)} = \bdelta^l$. So $a = \bdelta^l\lambda(k)$, and since $K^{\p}$ is closed under $\lambda$, $a \in K^{\p}$. In this case, we have $K^{\p}(a)^{\rc} = K^{\p}$, so $\lambda(K^{\p}(a)^{\rc}) = \lambda(K^{\p})$. In the second case, let $x \in (K^{\p}(a)^{\rc})^{>0}$. Again using the valuation inequality, there are $k \in K^{\p}$, $N \in \mbb{N}$, and $q \in \mbb{Q}$ such that $\frac{1}{N} < \frac{x}{ka^q} < N$. By assumption, $a \in A$. Therefore, a similar calculation as in \cref{case2-fundamental-lemma} above shows that $\lambda(x) \in A^{\p}\ip{a}_A$. Since $a \in A$, $\lambda(a) = a$, so we have $\lambda((K^{\p}(a)^{\rc})^{>0}) \subseteq \lambda((K^{\p})^{>0})\ip{\lambda(a)}_A$. The other inclusion is similar to the second part of \cref{case2-fundamental-lemma} above.
\end{proof}

We will also use the following lemma to prove density of certain subsets of $K$ when $K$ is equipped with the order topology.

\begin{lem}\label{density-lemma}
Let $\mc{M} \models T$. Let $f_1,\ldots,f_n: (GA)^m \to K$ be functions ($n,m \geq 1$) which are definable in the language $\lor(K)$. Then $K \setminus \bigcup_{j=1}^n f_j((GA)^m)$ is dense in $K$.
\end{lem}

\begin{proof}
This follows directly from Corollary 2.10 in \cite{AyhanPhilipp} once we show that $GA$ is small in $K$. Since $\mc{M} \models T$, $GA$ has the Mann property. Therefore, by Proposition 1.1 in van den Dries and G\"unayd\i n \cite{LouAyhan} and Proposition 2.9 in \cite{AyhanPhilipp}, $GA$ is small in $K(i)$. By Lemma 2.8 in \cite{AyhanPhilipp}, $GA$ is small in $K$.
\end{proof}

\subsection{Substructures of models of $T$}

Let $\mc{M} := (K,G,A,(\gamma)_{\gamma \in \Gamma},(\delta)_{\delta \in \Delta})$ be a model of $T$. Let $\kappa$ be an infinite cardinal such that $\kappa > \av{\GD}$.

Note that we can consider $A$ as a subgroup of $K(i)^{\times}$ by identifying the element $a$ of $A$ with the element $a+0i$ of $K(i)$. Letting 1 denote the identity of $G$, we have $G \cap A = \{1\}$ when $A$ is considered as a subgroup of $K(i)^{\times}$.

\begin{defn}\label{defn-sub-kga} Let $\Sub(K,G,A)$ be the collection of $\lor(P,V)$-structures $(K^{\prime},G^{\prime},A^{\prime})$ such that:
\begin{enumerate}
\item $K^{\prime}$ is a real closed subfield of $K$ of cardinality less than $\kappa$
\item $G^{\prime}$ is a pure subgroup of $G$ containing $\Gamma$
\item $A^{\prime}$ is a pure subgroup of $A$ containing $\Delta$
\item $K^{\prime}(i)$ and $\mbb{Q}(GA)$ are free over $\mbb{Q}(G^{\prime}A^{\prime})$
\item For all $k \in (K^{\prime})^{>0}$, there is $a \in A^{\prime}$ such that $a \leq k < a \bdelta$.
\end{enumerate}
\end{defn}

Note that in particular, we require that $G^{\p} \subseteq K^{\p}(i)$ and $A^{\p} \subseteq K^{\p}$.

\begin{lem}\label{lem-sub-kga}
If $(K^{\prime},G^{\prime},A^{\prime})$ satisfies conditions (1)-(4) in \cref{defn-sub-kga}, then $(K^{\prime},G^{\prime},A^{\prime})$ is indeed a substructure of $(K,G,A)$. Moreover, if $(K^{\prime},G^{\prime},A^{\prime})$ satisfies conditions (1)-(4) in \cref{defn-sub-kga}, then satisfying condition (5) is equivalent to $K^{\prime}$ being closed under $\lambda$.
\end{lem}

\begin{proof}[Proof of \cref{lem-sub-kga}]
Suppose $(K^{\p},G^{\p},A^{\p})$ satisfies conditions (1)-(4). We want to show that $G \cap K^{\p}(i) = G^{\p}$ and $A \cap K^{\p} = A^{\p}$. Since $G^{\p}$ is a subgroup of $G$, it is clear that $G^{\p} \subseteq G \cap K^{\p}(i)$. Now let $g \in K^{\p}(i) \cap G$. Then in particular, $g$ is algebraic over $\mbb{Q}(G)$, so $g$ is algebraic over $\mbb{Q}(GA)$. By condition (4) in \cref{defn-sub-kga}, $K^{\p}(i)$ and $\mbb{Q}(GA)$ are free over $\mbb{Q}(G^{\p}A^{\p})$, so $g$ is algebraic over $\mbb{Q}(G^{\p}A^{\p})$. Let $p(x) \in \mbb{Q}(G^{\p}A^{\p})[x]$ be a polynomial such that $p(g) = 0$ and let $d = \deg(p)$. By assumption, $\mc{M} \models T$, so in particular, $(K,GA,(\gamma\delta)_{\gamma \in \Gamma,\delta \in \Delta})$ satisfies the Mann axioms for $\GD$. Therefore, we may apply Lemma 5.12 in \cite{LouAyhan} to conclude that $g^d \in G^{\p}A^{\p}$. Thus, there are $h \in G^{\p}, a \in A^{\p}$ such that $g^d = ha$, so $a = h^{-1}g^d$. Since $G \cap A = \{1\}$, we must have $a = 1$, so $g^d = h$. By purity of $G^{\p}$ in $G$ (condition (2) in \cref{defn-sub-kga}), there is $g^{\p} \in G^{\p}$ such that $g^d = h = (g^{\p})^d$. Since $\mc{M} \models T$, in particular, $G_{\tor} = \Gamma_{\tor}$; thus, there is $\gamma \in \Gamma$ such that $g = g^{\p}\gamma$. Since condition (2) of \cref{defn-sub-kga} gives us that $\Gamma \subseteq G^{\prime}$, we have $g \in G^{\p}$.

The proof that $A \cap K^{\p} = A^{\p}$ is similar, using condition (3) in \cref{defn-sub-kga}. Note that if $a^d = (a^{\p})^d$ for $a^{\p},a \in A$, we automatically have $a = a^{\p}$. This is because $A \subseteq K^{>0}$ and $K$ is a real closed field.

Now suppose $(K^{\p},G^{\p},A^{\p})$ satisfies conditions (1)-(4) in \cref{defn-sub-kga}. Since $(K^{\p},G^{\p},A^{\p})$ satisfies conditions (1)-(4), $(K^{\p},G^{\p},A^{\p})$ is a substructure of $(K,G,A)$. Therefore, $(K^{\p}) \cap A = A^{\p}$; moreover, since $A^{\p} \subseteq (K^{\p})^{>0}$, $(K^{\p})^{>0} \cap A = A^{\p}$.

Suppose that in addition, $(K^{\p},G^{\p},A^{\p})$ satisfies condition (5) in \cref{defn-sub-kga}. We want to show that $K^{\p}$ is closed under $\lambda$; that is, we want to show that $\lambda((K^{\p})^{>0}) = K^{\p} \cap A = A^{\p}$. Let $k \in (K^{\p})^{>0}$. By condition (5) in \cref{defn-sub-kga}, there is $a^{\p} \in A^{\p}$ such that $a^{\p} \leq k < a^{\p}\bdelta$. By definition of $\lambda$, we have $\lambda(k) = a^{\p}$, so $\lambda(k) \in A^{\p}$. Conversely, let $a^{\p} \in A^{\p}$. We must have $\lambda(a^{\p}) = a^{\p}$ by definition of $\lambda$, so it is clear that $A^{\p} \subseteq \lambda((K^{\p})^{>0})$. Therefore, $K^{\p}$ is closed under $\lambda$.

Conversely, suppose that in addition to $(K^{\p},G^{\p},A^{\p})$ satisfying conditions (1)-(4) in \cref{defn-sub-kga}, $K^{\p}$ is closed under $\lambda$. Let $k^{\p} \in (K^{\p})^{>0}$. Since we assume $K^{\p}$ is closed under $\lambda$, we have $\lambda((K^{\p})^{>0}) = (K^{\p})^{>0} \cap A = A^{\p}$. Therefore, $\lambda(k^{\p}) \in A^{\p}$, so $\lambda(k^{\p})$ is an element of $A^{\p}$ such that $\lambda(k^{\p}) \leq k^{\p} < \lambda(k^{\p})\bdelta$. Thus, condition (5) of \cref{defn-sub-kga} is fulfilled.
\end{proof}

Note that if $(K^{\p},G^{\p},A^{\p})$ is an $\lor(P,V)$-structure satisfying conditions (1) and (5) in \cref{defn-sub-kga}, then $A^{\p}$ is pure in $A$. However, in practice, to show that an $\lor(P,V)$-structure $(K^{\p},G^{\p},A^{\p})$ is in $\Sub(K,G,A)$, we will first check conditions (1)-(4) to show that $(K^{\p},G^{\p},A^{\p})$ is a substructure of $(K,G,A)$, then show that $K^{\p}$ is closed under $\lambda$.\\

We will now state two facts about freeness that we will use later on. Both are surely known, but we include proofs for the reader's convenience.

\begin{lem}\label{freeness-lemma}
Let $K^{\p}$ be a subfield of $K$, $G^{\p}$ a subgroup of $G$, and $A^{\p}$ a subgroup of $A$ such that $G^{\p}A^{\p} \subseteq K^{\p}(i)$ and $K^{\p}(i)$ and $\mbb{Q}(GA)$ are free over $\mbb{Q}(G^{\p}A^{\p})$. Let $E$ be a subset of $G$ or of $A$, and let $X \subseteq K$ be a subset which is algebraically independent over $K^{\p}(GA)$. Then $K^{\p}(\Re(E),X)^{\rc}(i)$ and $\mbb{Q}(GA)$ are free over $\mbb{Q}(G^{\p}A^{\p},E)$.
\end{lem}

\begin{proof}
Since $K^{\p}(i)$ and $\mbb{Q}(GA)$ are free over $\mbb{Q}(G^{\p}A^{\p})$, we have $G^{\p} \subseteq K^{\p}(i)$ and $A^{\p} \subseteq K^{\p}$. If $E \subseteq G$, then for any $z \in E$, we have $\Re(z) = \frac{z^2+1}{2z}$ since $G \subseteq \mbb{S}^1(K)$. Thus, $\Re(E) \subseteq \mbb{Q}(E)$. If $E \subseteq A$, then $\Re(E) = E$ since $A \subseteq K$. Therefore, in both cases, we have $\Re(E) \subseteq \mbb{Q}(E)$.

The conclusion of the lemma follows from the fact that $\Re(E) \subseteq \mbb{Q}(E)$, together with \cref{prop12-bourbaki} and \cref{exercise14-bourbaki}.

\end{proof}

\begin{cor}\label{freeness-cor}
Let $G^{\p}$ be a subgroup of $G$ and $A^{\p}$ a subgroup of $A$. Let $k$ be a subfield of $K$ such that $k \subseteq \mbb{Q}(\Re(G^{\p}A^{\p}))$ and let $X \subseteq K$. Suppose that $X$ is algebraically independent over $\mbb{Q}(GA)$. Then $k(X)^{\rc}(i)$ and $\mbb{Q}(GA)$ are free over $\mbb{Q}(G^{\p}A^{\p})$.
\end{cor}

\begin{proof}
First note that for $g \in G^{\p}$ and $a \in A^{\p}$, $\Re(ga) = \frac{(ga)^2+a^2}{2ga}$ since $G^{\p} \subseteq \mbb{S}^1(K)$. Therefore, $\Re(G^{\p}A^{\p}) \subseteq \mbb{Q}(G^{\p}A^{\p})$. Since no (nonempty) subset of $k(i)$ is algebraically independent over $\mbb{Q}(G^{\p}A^{\p})$, $k(i)$ and $\mbb{Q}(GA)$ are free over $\mbb{Q}(G^{\p}A^{\p})$. Note that $k(GA) \subseteq \mbb{Q}(GA)$. Since we assume $X$ is algebraically independent over $\mbb{Q}(GA)$, $X$ is also algebraically independent over $k(GA)$. Applying \cref{freeness-lemma} with $E = \emptyset$, we see that $k(X)^{\rc}(i)$ and $\mbb{Q}(GA)$ are free over $\mbb{Q}(G^{\p}A^{\p})$.
\end{proof}

Since $\emptyset$ is considered to be algebraically independent over any field, we will sometimes apply \cref{freeness-lemma} and \cref{freeness-cor} with $X = \emptyset$.

\section{Proof of \cref{elequiv-statement}}\label{sec:elequiv}
In this section we establish \cref{elequiv-statement}. In fact, we prove the following slightly more general result.

\begin{thm}\label{main-thm}
Let $(K,G,A,(\gamma)_{\gamma \in \Gamma},(\delta)_{\delta \in \Delta})$ and $(L,H,B,(\gamma)_{\gamma \in \Gamma},(\delta)_{\delta \in \Delta})$ be two models of $T$. Then they are elementarily equivalent if and only if $[p]G = [p]H$ for all primes $p$, and for all $\gamma \in \Gamma$ and all $n > 0$:
\[
\gamma \text{ is an } n\text{th power in } G \text{ if and only if } \gamma \text{ is an } n\text{th power in } H.
\]
\end{thm}

First suppose $(K,G,A,(\gamma)_{\gamma \in \Gamma},(\delta)_{\delta \in \Delta}) \equiv (L,H,B,(\gamma)_{\gamma \in \Gamma},(\delta)_{\delta \in \Delta})$. The statements ``$[p]G = n$'' and ``$\gamma$ is/is not an $n$th power in $G$'' are first-order sentences in our language. We can also express the statement ``$[p]G = \infty$" using first-order sentences in our language. Thus, the ``only if'' direction of the theorem statement is clear.\\

We now prove the other direction of \cref{main-thm}. Fix two models
\[
\mc{M} := (K,G,A,(\gamma)_{\gamma \in \Gamma},(\delta)_{\delta \in \Delta}) \hbox{ and } \mc{N} := (L,H,B,(\gamma)_{\gamma \in \Gamma},(\delta)_{\delta \in \Delta})
\]
of $T$ such that $[p]G = [p]H$ for all primes $p$, and for all $\gamma \in \Gamma$ and all $n > 0$:
\[
\gamma \text{ is an } n\text{th power in } G \text{ if and only if } \gamma \text{ is an } n\text{th power in } H.
\]
We want to prove that $\mc{M} \equiv \mc{N}$. We may assume that $\mc{M}, \mc{N}$ are $\kappa$-saturated for some infinite $\kappa > \av{\GD}$. Let $\mathcal{I}$ be the collection of isomorphisms between members of $\Sub(K,G,A)$ and $\Sub(L,H,B)$ that fix $\Delta$ and $\Gamma$ pointwise. We will show that $\mathcal{I}$ is a nonempty back-and-forth system, which will give us that $\mc{M} \equiv \mc{N}$.

\subsection{$\mc{I}$ is nonempty}\label{subsec:i-nonempty}

To see that $\mc{I}$ is nonempty, let
\[
K^{\p} = \mbb{Q}(\RGD)^{\rc}, G^{\p} = \{g \in G : g^n \in \Gamma \text{ for some } n > 0\}, A^{\p} = \Delta
\]
and let
\[
L^{\p} = \mbb{Q}(\RGD)^{\rc}, H^{\p} = \{h \in H : h^n \in \Gamma \text{ for some } n > 0\}, B^{\p} = \Delta.
\]

We must first check that $K^{\p} \subseteq K$, $G^{\p} \subseteq (K^{\p})^2$, and $A^{\p} \subseteq K^{\p}$. It is clear that $K^{\p} \subseteq K$. If $g \in G^{\p}$, then there is $n > 0$ such that $g^n \in \Gamma$. Let $z = g^n$, $a = \Re(z)$, and $b = \Im(z)$. Since $z \in \Gamma$, $a^2+b^2 = 1$, so it can be checked that $z^2-2az+1 = 0$. Therefore, $g$ is algebraic over $\mbb{Q}(\RGD)^{\rc}$, and so $g \in K^{\p}(i)$.

We now check that $(K^{\p},G^{\p},A^{\p}) \in \Sub(K,G,A)$. By \cref{lambda-closure-lem2}, $A^{\p}$ is pure in $A$. We apply \cref{freeness-cor} with $X = \emptyset$ and $k = \mbb{Q}(\RGD)$ to get that $K^{\p}(i)$ and $\mbb{Q}(GA)$ are free over $\mbb{Q}(G^{\p}A^{\p})$. Now we must show that $K^{\p}$ is closed under $\lambda$. Since $(K^{\p},G^{\p},A^{\p})$ satisfies conditions (1)-(4) of \cref{defn-sub-kga}, $K^{\p} \cap A = A^{\p}$. By \cref{lambda-closure-lem2} and \cref{lambda-closure-lem3}, $\lambda((K^{\p})^{>0}) = \Delta = A^{\p}$, so $K^{\p}$ is closed under $\lambda$.

The proof that $(L^{\p},H^{\p},B^{\p}) \in \Sub(L,H,B)$ is similar.

We now show that there is a function $f: (K^{\p},G^{\p},A^{\p}) \to (L^{\p},H^{\p},B^{\p})$ with $f \in \mc{I}$. To prove this, let $f: K^{\p} \to L^{\p}$ be the obvious function extending the identity map on $\Re(\GD)$. Let $p_1$ be the set of $\lor$-formulas satisfied by elements of $\Re(\GD)$ in $\mc{M}$, and let $p_2$ be the set of $\lor$-formulas satisfied by elements of $\Re(\GD)$ in $\mc{N}$. Since we assume that $\mc{M}$ and $\mc{N}$ satisfy the orientation axioms for $\GD$, we have $p_1 = p_2$. Therefore, $f$ is an ordered field isomorphism. Clearly, $f(A^{\p}) = B^{\p}$. Similarly, by our assumption that $\gamma$ is an $n$th power in $G$ if and only if $\gamma$ is an $n$th power in $H$, $f(G^{\p}) = H^{\p}$. (Let $g: K^2 \to L^2$ be the function defined by $g(k_1,k_2) = (f(k_1),f(k_2))$. By $f(G^{\p})$, we mean $g(G^{\p})$.) Clearly $f$ fixes $\GD$ pointwise.

Therefore, $\mc{I}$ is nonempty.

\subsection{$\mc{I}$ is a back-and-forth system} Let $(K^{\p},G^{\p},A^{\p}) \in \Sub(K,G,A)$ and $(L^{\p},H^{\p},B^{\p}) \in \Sub(L,H,B)$. Let $\iota: (K^{\p},G^{\p},A^{\p}) \to (L^{\p},H^{\p},B^{\p})$ be in $\mc{I}$, and let $a \in K \setminus K^{\p}$. We have four cases:
\begin{enumerate}
\item $a \in A$
\item $a \in \Re(G)$ or $a \in \Im(G)$
\item $a \in K^{\p}(\Re(GA) \cup \Im(GA))^{\rc}$
\item $a \in K \setminus K^{\p}(\Re(GA) \cup \Im(GA))^{rc}$
\end{enumerate}

\begin{case-mainthm}\label{case1-mainthm} $a \in A$. \end{case-mainthm}

Define sets $\Sigma_1$, $\Sigma_2$ of $\lom(V,K^{\p})$-formulas in the variable $x$ by
\[
\Sigma_1 := \{\iota(k_1) < x < \iota(k_2) : a \in (k_1,k_2), k_1,k_2 \in K^{\p}\},
\]
\[
\Sigma_2 := \{\iota(a^{\p})x^l \in B^{[m]} : a^{\p} \in A^{\p}, l \in \mbb{Z}, m > 0, a^{\p}a^l \in A^{[m]}\}.
\]
Our first step is to find $b \in B$ such that $b$ satisfies $\Sigma_1 \cup \Sigma_2$.

Since $(K^{\p},G^{\p},A^{\p}) \in \Sub(K,G,A)$, $A^{\p}$ is pure in $A$. Therefore, $[p]A^{\p} \leq [p]A$ for all primes $p$. Moreover, since $(K^{\p},G^{\p},A^{\p}) \in \Sub(K,G,A)$, $\Delta \subseteq A^{\p}$. Therefore, by the last axiom in $T$, $[p]A^{\p} \geq [p]A$. So we have $[p]A^{\p} = [p]A$ for all primes $p$. Similarly, we have $[p]B^{\p} = [p]B$ for all primes $p$. Since $A$ and $B$ are regularly discrete, we have $[p]A = [p]B = p$ for all primes $p$ by \cref{reg-discrete-lem}. Since $\mc{M},\mc{N}$ are $\kappa$-saturated (where $\kappa > \av{K^{\p}}$), $(A,A^{\p})$ and $(B,B^{\p})$ are $\kappa$-saturated. Therefore, we may apply Lemma 4.2.1 in \cite{Ayhan}. Using this lemma, we fix $h \in B$ such that for all $a^{\p} \in A^{\p}$, $m > 0$, and $l \in \mbb{Z}$,
\[
a^{\p}a^l \in A^{[m]} \text{ if and only if } \iota(a^{\p})h^l \in B^{[m]}.
\]

In order to prove that $\Sigma_1 \cup \Sigma_2$ is satisfiable by an element of $B$, we show that $\Sigma_1 \cup \Sigma_2$ is finitely satisfiable by an element of $B$. For if we can show this, then $\Sigma_1 \cup \Sigma_2$ is satisfied by an element $b \in B$ by $\kappa$-saturation of $(L,H,B)$.

To prove that $\Sigma_1 \cup \Sigma_2$ is finitely satisfiable by an element of $B$, it suffices to show that for given $k_1,k_2 \in K^{\p}$ such that $k_1 < a < k_2$, there exists $\beta \in B$ such that $\iota(k_1) < \beta < \iota(k_2)$ and $\beta$ satisfies $\Sigma_2$. Thus, we fix $k_1,k_2 \in K^{\p}$ such that $k_1 < a < k_2$. We may assume that $k_1,k_2 > 0$ since $a \in A \subseteq K^{>0}$.

Since $K^{\p}$ is closed under $\lambda$ and $a \notin A^{\p}$, we have
\[
k_1 < \lambda(k_1)\bdelta < \lambda(k_1)\bdelta^2 < \ldots < a < k_2.
\]
Therefore, the interval contains infinitely many elements of $A^{\p}$.

Since $\iota: (K^{\p},G^{\p},A^{\p}) \to (L^{\p},H^{\p},B^{\p})$ is an isomorphism, the interval $(\iota(k_1),\iota(k_2))$ in $L^{\p}$ contains infinitely many elements of $B^{\p}$. Since $h \in B$, the interval $I := (\iota(k_1)h^{-1},\iota(k_2)h^{-1})$ contains infinitely many elements of $B$.

Consider the set of $\lor(V,B)$-formulas
\[
\{x \in (\iota(k_1)h^{-1},\iota(k_2)h^{-1}) \wedge \exists y \in B (x = y^k) : k > 0\}.
\]
By $\kappa$-saturation, to find an element satisfying this set of formulas, it suffices to find an element $x \in I \cap \bigcap_{i=1}^s B^{[n_i]}$ for arbitrary $n_1,\ldots,n_s \geq 1$ and $s \geq 1$. If we are given $n_1,\ldots,n_s$, let $n = n_1 \ldots n_s$. Since there are infinitely many elements of $B$ in $I$, let $b_1,b_2$ be elements of $I \cap B$ such that there are at least $n$ elements of $B$ in the interval $(b_1,b_2)$. Since $B$ is assumed to be regularly discrete, there is $\eta^{\p} \in B^{\p}$ with $\eta^{\p} \in (b_1,b_2) \cap B^{[n]}$ by \cref{reg-discrete-lem}. Therefore, $\eta^{\p} \in I \cap \bigcap_{i=1}^s B^{[n_i]}$.

Let $\eta \in I \cap \bigcap_{k \geq 1} B^{[k]}$. In particular, $\eta$ is divisible in $B$ by all $k \geq 1$. Let $\beta = h \eta^{1/n}$. Note that $h^n \eta \in (\iota(k_1),\iota(k_2))$, so $\beta^n \in (\iota(k_1),\iota(k_2))$.

It follows from our choice of $\beta$ that for all $a^{\p} \in A^{\p},l \in \mbb{Z}$, and $m > 0$, we have
\[
a^{\p}a^l \in A^{[m]} \text{ if and only if } \iota(a^{\p})\beta^l \in B^{[m]}.
\]

Therefore, $\iota(k_1) < \beta < \iota(k_2)$ and $\beta$ satisfies $\Sigma_2$.

We now have $b \in B$ such that $\mc{N} \models \Sigma_1(b) \cup \Sigma_2(b)$. Since $b$ satisfies the same cut over $K^{\p}$ that $a$ does over $L^{\p}$, we have an $\lor$-isomorphism $\iota^{\p}: K^{\p}(a)^{\rc} \to L^{\p}(b)^{\rc}$ extending $\iota$ which takes $a$ to $b$. Since $b$ satisfies $\Sigma_2$, we have $\iota^{\p}(A^{\p}\ip{a}_A) = B^{\p}\ip{b}_B$. To check that $\iota^{\p} \in \mc{I}$, we must check that
\[
(K^{\p}(a)^{\rc},G^{\p},A^{\p}\ip{a}_A) \in \Sub(K,G,A).
\]
In particular, we must show that $(K^{\p}(a)^{\rc})(i)$ and $\mbb{Q}(GA)$ are free over $\mbb{Q}(G^{\p}A^{\p}\ip{a}_A)$. By assumption, $K^{\p}(i)$ and $\mbb{Q}(GA)$ are free over $\mbb{Q}(G^{\p}A^{\p})$. Therefore, we may apply \cref{freeness-lemma} with $E = \{a\}$ and $X = \emptyset$ to get that $K^{\p}(a)^{\rc}(i)$ and $\mbb{Q}(GA)$ are free over $\mbb{Q}(G^{\p}A^{\p},a)$. Since $\mbb{Q}(G^{\p}A^{\p},a) \subseteq \mbb{Q}(G^{\p}A^{\p}\ip{a}_A)$, $K^{\p}(i)$ and $\mbb{Q}(GA)$ are free over $\mbb{Q}(G^{\p}A^{\p}\ip{a}_A)$.

We now want to prove that $K^{\p}(a)^{\rc}$ is closed under $\lambda$. Since $(K^{\p}(a)^{\rc},G^{\p},A^{\p}\ip{a}_A)$ is a substructure of $(K,G,A)$, it suffices to prove that $\lambda((K^{\p}(a)^{\rc})^{>0}) = A^{\p}\ip{a}_A$. Since $a \in A$, in order to prove this, it suffices to prove that $v(a) \notin v(K^{\p})$ by the \nameref{fundamental-lemma}. Suppose for a contradiction that $v(a) \in v(K^{\p})$. By the proof of the \nameref{fundamental-lemma}, if $v(a) \in v(K^{\p})$, then $a \in K^{\p}$. But by assumption, $a \in K \setminus K^{\p}$, a contradiction. So we must have $v(a) \notin v(K^{\p})$.

The proof that $(L^{\p}(b)^{\rc},H^{\p},B^{\p}\ip{b}_B) \in \Sub(L,H,B)$ is similar.

\begin{case-mainthm}\label{case2-mainthm} Now suppose $a \in \Re(G)$. (The case where $a \in \Im(G)$ is similar.) \end{case-mainthm}
Let $A^{(1)} := A^{\p}$ and for $j = 1,2,\ldots$, let $A^{(j+1)} = \lambda((K^{\p}(A^{(j)},a)^{\rc})^{>0})$. Let $A^{\infty} := \bigcup_{j=1}^{\infty} A^{(j)}$. Note that $A^{(j)} \subseteq A^{(j+1)}$ for all $j$ by definition. Moreover, for each $j$, $\av{A^{(j)}} < \kappa$ since $\kappa > \av{K^{\p}}$. Therefore, $\av{A^{\infty}} < \kappa$.\\

Let $B^{(1)} := B^{\p}$. For $j \geq 1$, we recursively define $B^{(j)} \subseteq B$ and an ordered field isomorphism $f_j: K^{\p}(A^{(j)})^{\rc} \to L^{\p}(B^{(j)})^{\rc}$ such that $f_j \in \mc{I}$ and $f_j(A^{(j)}) = B^{(j)}$. In particular, we require that $(K^{\p}(A^{(j)})^{\rc},G^{\p},A^{(j)}) \in \Sub(K,G,A)$ and $(L^{\p}(B^{(j)})^{\rc},H^{\p},B^{(j)}) \in \Sub(L,H,B)$ for all $j$. Note that $K^{\p}(A^{(1)}) = K^{\p}$ and $L^{\p}(B^{(1)}) = B^{\p}$. Thus, we take $f_1: K^{\p}(A^{(1)})^{\rc} \to L^{\p}(B^{(1)})^{\rc}$ to be $\iota$. Now suppose we have defined $f_j$ and $B^{(j)}$ ($j \geq 1$), and we want to define $B^{(j+1)}$ and $f_{j+1}$. \\

Let $K_j = K^{\p}(A^{(j)})^{\rc}$ and let $L_j = L^{\p}(B^{(j)})^{\rc}$. Since $f_j \in \mc{I}$, $K_j$ is closed under $\lambda$. Therefore, $\lambda(K_j^{>0}) = A^{(j)}$. To define $B^{(j+1)}$ and $f_{j+1}$, we consider two cases: (1) for all $x \in (K_j(a)^{\rc})^{>0}$, $v(x) \in v(K_j)$, and (2) there is $z \in (K_j(a)^{\rc})^{>0}$ such that $v(z) \notin v(K_j)$.

First assume that for all $x \in (K_j(a)^{\rc})^{>0}$, $v(x) \in v(K_j)$. By Case 1 of the \nameref{fundamental-lemma}, $A^{(j+1)} = \lambda((K_j(a)^{\rc})^{>0}) = A^{(j)}$. So we take $B^{(j+1)} = B^{(j)}$ and $f_{j+1} = f_j$. \\

Now suppose that there is $z \in (K_j(a)^{\rc})^{>0}$ such that $v(z) \notin v(K_j)$.\\

 We show that in this case, $A^{(j+1)} = A^{(j)}\ip{\lambda(f(a))}_A$ for some $\lor(K_j)$-definable function $f$. Let $f$ be an $\lor(K_j)$-definable function such that $z = f(a)$. Then by Case 2 of the \nameref{fundamental-lemma}, we have $\lambda((K_j(a)^{\rc})^{>0}) = A^{(j)}\ip{\lambda(f(a))}_A$. (Here we must use the assumption that $\lambda(K_j^{>0}) = A^{(j)}$.) But since $K_j = K^{\p}(A^{(j)})^{\rc}$, we have $K_j(a)^{\rc} = K^{\p}(A^{(j)},a)^{\rc}$. Therefore,
\[
A^{(j+1)} = \lambda((K^{\p}(A^{(j)},a)^{\rc})^{>0}) = \lambda((K_j(a)^{\rc})^{>0}) = A^{(j)}\ip{\lambda(f(a))}_A.
\]
By definition, $K_{j+1} = K^{\p}(A^{(j+1)})^{\rc}$. Thus, by what we just proved, $K_{j+1} = K^{\p}(A^{(j)}\ip{\lambda(f(a))}_A)^{\rc}$. It can be shown that
\[
K^{\p}(A^{(j)}\ip{\lambda(f(a))}_A)^{\rc} = (K^{\p}(A^{(j)})(\lambda(f(a))))^{\rc}
\]
so $K_{j+1} = K_j(\lambda(f(a)))^{\rc}$.\\

By our inductive assumption, $f_j: (K_j,G^{\p},A^{(j)}) \to (L_j,H^{\p},B^{(j)})$ is in $\mc{I}$. Moreover, $\lambda(f(a)) \in A$ and $\lambda(f(a)) \notin K_j$. (Since $v(f(a)) = v(\lambda(f(a))$ and we assume $v(f(a)) \notin v(K_j)$, we cannot have $\lambda(f(a)) \in K_j$.) Therefore, we may apply \cref{case1-mainthm} of this theorem to find $b \in B$ and an ordered field isomorphism
\[
f_{j+1}: (K_j(\lambda(f(a)))^{\rc},G^{\p},A^{(j)}\ip{\lambda(f(a))}_A) \to (L_j(b)^{\rc},H^{\p},B^{(j)}\ip{b}_B)
\]
with $f_{j+1} \in \mc{I}$ taking $\lambda(f(a))$ to $b$. Thus, in this case, we take $B^{(j+1)}$ to be $B^{(j)}\ip{b}_B$. Note that $f_{j+1}$ extends $f_j$ by construction.

This completes the recursive construction. Now define
\[
f_{\infty} := \bigcup_{j \geq 1} f_j, B^{\infty} := \bigcup_{j \geq 1} B^{(j)}.
\]

We will now show that $(K^{\p}(A^{\infty})^{\rc},G^{\p},A^{\infty}) \in \Sub(K,G,A)$, $(L^{\p}(B^{\infty})^{\rc},H^{\p},B^{\infty}) \in \Sub(L,H,B)$, and $f_{\infty} \in \mc{I}$.\\

To show that $(K^{\p}(A^{\infty})^{\rc},G^{\p},A^{\infty}) \in \Sub(K,G,A)$, we first show that $A^{\infty}$ contains $\Delta$ and is pure in $A$. Since $A^{\p} \subseteq A^{\infty}$ and $A^{\p}$ contains $\Delta$ by assumption, $A^{\infty}$ also contains $\Delta$. The pureness of $A^{\infty}$ follows easily from the pureness of $A^{(N)}$ for each $N$.\\

We now check the freeness condition. By assumption, $K^{\p}(i)$ and $\mbb{Q}(GA)$ are free over $\mbb{Q}(G^{\p}A^{\p})$. By definition of $A^{\infty}$, we have $A^{\p} \subseteq A^{\infty}$. Therefore, we may apply \cref{freeness-lemma} with $E = A^{\infty}$ and $X = \emptyset$ to show that $K^{\p}(A^{\infty})^{\rc}(i)$ and $\mbb{Q}(GA)$ are free over $\mbb{Q}(G^{\p}A^{\infty})$.\\

We now want to check that $K^{\p}(A^{\infty})^{\rc}$ is closed under $\lambda$. In particular, we must show that $\lambda((K^{\p}(A^{\infty})^{\rc})^{>0}) \subseteq A^{\infty}$. Let $x \in (K^{\p}(A^{\infty})^{\rc})^{>0}$. Then $x \in (K^{\p}(A^{(N+1)})^{\rc}$ for some $N \in \mbb{N}$. By construction, we have $\lambda(x) \in A^{(N+2)}$. Therefore, $\lambda(x) \in A^{\infty}$.\\

The proof that $(L^{\p}(B^{\infty})^{\rc},H^{\p},B^{\infty}) \in \Sub(L,H,B)$ is similar, using the construction of $B^{\infty}$. By construction, $f_{\infty}$ is an ordered field isomorphism between $K^{\p}(A^{\infty})^{\rc}$ and $L^{\p}(B^{\infty})^{\rc}$ taking $A^{\infty}$ to $B^{\infty}$. Since $f_{\infty}$ extends $\iota$, it fixes $\Gamma$ and $\Delta$, so $f_{\infty} \in \mc{I}$.\\

Our next step is to find $\iota^{\p} \in \mc{I}$ such that $\iota^{\p}$ extends $f_{\infty}$ and $a$ is in the domain of $\iota^{\p}$. Since $a \in \Re(G)$, let $g \in G$ with $a = \Re(g)$. By assumption, $G$ and $H$ are regularly dense oriented abelian groups and $[p]G = [p]H$ for all primes $p$. Since $\mc{M},\mc{N} \models T$, we also have $G_{\tor} = G^{\p}_{\tor}$ and $H_{\tor} = H^{\p}_{\tor}$. Moreover, $G^{\p}$ is pure in $G$ and $H^{\p}$ is pure in $H$. Thus, we apply \cref{claim-812} to obtain $\eta \in H$ and an oriented group isomorphism $j: G^{\p}\ip{g}_G \to H^{\p}\ip{\eta}_H$ taking $g$ to $\eta$ and extending $f_{\infty}$. (That is, for $(\alpha,\beta) \in G^{\p}$, $j(\alpha,\beta) = (f_{\infty}(\alpha),f_{\infty}(\beta))$.)\\

We now find $h \in H$ such that $h$ satisfies the set of $\lor(L^{\p},P)$-formulas $S_1 \cup S_2 \cup S_3$ in the variable $x$, where
\[
S_1 = \{f_{\infty}(k_1) < \Re(x) < f_{\infty}(k_2) : k_1 < \Re(g) < k_2, k_1,k_2 \in K^{\p}(A^{\infty})^{\rc}\},
\]
\[
S_2 = \{f_{\infty}(g^{\p})x^l \in H^{[m]} : g^{\p}g^l \in G^{[m]}, g^{\p} \in G^{\p}, l \in \mbb{Z}, m > 0\},
\]
\[
S_3 = \{f_{\infty}(g^{\p})x^l \notin H^{[m]} : g^{\p}g^l \notin G^{[m]}, g^{\p} \in G^{\p}, l \in \mbb{Z}, m > 0\}
\]

In order to find an element of $H$ satisfying $S_1 \cup S_2 \cup S_3$, it suffices by $\kappa$-saturation of $\mc{N}$ to show that every finite subset of $S_1 \cup S_2 \cup S_3$ is realized by an element of $H$. As in \cref{case1-mainthm}, it suffices to find $z \in H$ realizing a single formula from $S_1$ such that $z$ also satisfies $S_2$ and $S_3$.\\

Thus, let $k_1,k_2 \in K^{\p}(A^{\infty})^{\rc}$ be such that $k_1 < \Re(g) < k_2$. We may assume without loss of generality that $k_1,k_2 \in [-1,1]$. Let $y_1,y_2\in K$ such that $y_1 = (1 - (f_{\infty}(k_1))^2)^{1/2}$, $y_2 = (1 - (f_{\infty}(k_2))^2)^{1/2}$. (In particular, $y_1,y_2 > 0$.) Let $z_1 = (f_{\infty}(k_1),y_1)$ and let $z_2 = (f_{\infty}(k_2),y_2)$. Thus, $y_1,y_2$ are elements of $L^{\p}(B^{\infty})^{\rc}$ such that $z_1,z_2 \in \mbb{S}^1(L)$. Since $k_1 < k_2$, $\mc{O}(1,z_2,z_1)$ holds in $\mc{M}$. Let
\[
I = \{z \in \mbb{S}^1(L) : \mc{M} \models \mc{O}(z_2\eta^{-1},z,z_1\eta^{-1})\}.
\]
That is, $I$ is the "interval" in $\mbb{S}^1(L)$ between $z_2\eta^{-1}$ and $z_1\eta^{-1}$.\\

We claim that there is $z \in \bigcap_{m =1}^{\infty} H^{[m]}$ such that $z \in I$. By $\kappa$-saturation of $\mc{N}$, it suffices to find $z^{\p} \in I$ with $z^{\p} \in \bigcap_{j=1}^n H^{[m_j]}$ for arbitrary $m_1,\ldots,m_n \geq 1$, $n \geq 1$. If we are given $m_1,\ldots,m_n$, let $m = m_1 \ldots m_n$. By regular density of $H$ in $\mbb{S}^1(L)$, there is $z^{\p} \in H^{[m]}$ such that $z^{\p} \in I$. Since $m = m_1 \ldots m_n$, we also have $z \in H^{[m_j]}$ for $1 \leq j \leq n$.\\

Thus, let $z$ be an element of $I$ with $z \in \bigcap_{m=1}^{\infty} H^{[m]}$. By definition of $I$, $\mc{N} \models \mc{O}(z_2,z\eta,z_1)$. In particular, $f_{\infty}(k_1) < \Re(z\eta) < f_{\infty}(k_2)$ holds.\\

We now show that for all $g^{\p} \in G^{\p}$, $l \in \mbb{Z}$, $m > 0$ such that $g^{\p}g^l \in G^{[m]}$, we have $f_{\infty}(g^{\p})(z\eta)^l \in H^{[m]}$. By our choice of $z$, we have $z \in H^{[m]}$, so we also have $z^l \in H^{[m]}$. Since $j$ is an oriented group isomorphism extending $f_{\infty}$ and taking $g$ to $\eta$, we have $f_{\infty}(g^{\p})\eta^l \in H^{[m]}$. Therefore, $f_{\infty}(g^{\p})(z\eta)^l \in H^{[m]}$.\\

We must also show that for all $g^{\p} \in G^{\p}$, $l \in \mbb{Z}$, $m > 0$ such that $g^{\p}g^l \notin G^{[m]}$, we have $f_{\infty}(g^{\p})(z\eta)^l \notin H^{[m]}$. Suppose $f_{\infty}(g^{\p})(z\eta)^l \in H^{[m]}$. By our choice of $z$, we have $z^{-l} \in H^{[m]}$. Therefore, $f_{\infty}(g^{\p})\eta^l \in H^{[m]}$. Since $j$ is an oriented group isomorphism extending $f_{\infty}$ and taking $g$ to $\eta$, we have $g^{\p}g^l \in G^{[m]}$.\\

So $z\eta$ satisfies every formula in $S_2$ and $S_3$, as well as the formula $f_{\infty}(k_1) < \Re(z\eta) < f_{\infty}(k_2)$ for our given $k_1,k_2$.\\

By $\kappa$-saturation of $\mc{N}$, we have $h \in H$ such that $h$ satisfies $S_1 \cup S_2 \cup S_3$. Since $\Re(h)$ satisfies the same cut over $L^{\p}(B^{\infty})$ that $a$ does over $K^{\p}(A^{\infty})$, we can extend $f_{\infty}$ to an ordered field isomorphism $\iota^{\p}: K^{\p}(A^{\infty},a)^{\rc} \to L^{\p}(B^{\infty},\Re(h))^{\rc}$ taking $a$ to $\Re(h)$. Moreover, since $\iota^{\p}$ is an ordered field isomorphism, $\iota^{\p}(g) = h$. Thus, we have $\iota^{\p}(G^{\p}\ip{g}_G) = H^{\p}\ip{h}_H$ by our choice of $h$. Since $\iota^{\p}$ extends $f_{\infty}$, we also have $\iota^{\p}(A^{\infty}) = B^{\infty}$.\\

We now show that $(K^{\p}(A^{\infty},a)^{\rc},G^{\p}\ip{g}_G, A^{\infty}) \in \Sub(K,G,A)$. We first check the freeness condition. As proved above, $K^{\p}(A^{\infty})^{\rc}$ and $\mbb{Q}(GA)$ are free over $\mbb{Q}(G^{\p}A^{\infty})$. Therefore, we apply \cref{freeness-lemma} with $E = \{g\}$ and $X = \emptyset$ to get that $K^{\p}(A^{\infty},a)^{\rc}(i)$ and $\mbb{Q}(GA)$ are free over $\mbb{Q}(G^{\p}A^{\infty},g)$. Since $\mbb{Q}(G^{\p}A^{\infty},g) \subseteq \mbb{Q}(G^{\p}\ip{g}_GA^{\infty})$, we see that $K^{\p}(A^{\infty},a)^{\rc}(i)$ and $\mbb{Q}(GA)$ are free over $\mbb{Q}(G^{\p}\ip{g}_GA^{\infty})$.\\

Next, we want to show that $\lambda((K^{\p}(A^{\infty},a)^{\rc})^{>0}) = A^{\infty}$. Let $x \in (K^{\p}(A^{\infty},a)^{\rc})^{>0}$. Then $x = \sigma(k,c,a)$ for some $\lor$-definable function $\sigma$, some tuple $k$ of elements of elements of $K^{\p}$, and some tuple $c$ of elements of $A^{\infty}$. Since $c$ is a tuple of elements from $A^{\infty}$, we must have $c \subseteq A^{(j)}$ for some $j \geq 1$. We have $\lambda(\sigma(k,c,a)) \in A^{(j+1)}$ by definition of $A^{(j+1)}$, so $\lambda(x) \in A^{\infty}$.\\

Let $b = \Re(h)$. The proof that $(L^{\p}(B^{\infty},b)^{\rc},H^{\p}\ip{h}_H,B^{\infty}) \in \Sub(L,H,B)$ is mostly similar to the proof that $(K^{\p}(A^{\infty},a)^{\rc},G^{\p}\ip{g}_G,A^{\infty}) \in \Sub(K,G,A)$. We only need to show that $L^{\p}(B^{\infty},b)^{\rc}$ is closed under $\lambda$.

Let $\sigma(\ell,d,b) \in (L^{\p}(B^{\infty},b)^{\rc})^{>0}$, where $\sigma$ is an $\lor$-definable function, $\ell$ is a tuple of elements from $L^{\p}$, and $d$ is a tuple of elements from $B^{\infty}$. Let $y = \lambda(\sigma((\iota^{\p})^{-1}(\ell), (\iota^{\p})^{-1}(d),a))$, so that $y \in A^{\infty}$. By definition of $\lambda$,
\[
y \leq \sigma((\iota^{\p})^{-1}(\ell), (\iota^{\p})^{-1}(d),a) \leq \bdelta y.
\]
Since $\iota^{\p}$ is an isomorphism taking $a$ to $b$ and fixing $\Delta$, $\iota^{\p}(y) \leq \sigma(\ell,d,b) \leq \bdelta\iota^{\p}(y)$. Therefore, $\lambda(x) = \iota^{\p}(y)$. Since $\iota^{\p}(A^{\infty}) = B^{\infty}$, we have $\lambda(x) \in B^{\infty}$. Therefore, $L^{\p}(B^{\infty},b)^{\rc}$ is closed under $\lambda$.

Therefore, $\iota^{\p}$ is an element of $\mc{I}$ extending $\iota$ with $a$ in its domain.

\begin{case-mainthm}\label{case3-mainthm} Suppose $a \in K^{\p}(\Re(GA) \cup \Im(GA))^{\rc}$. \end{case-mainthm}
Since $a \in K^{\p}(\Re(GA) \cup \Im(GA))^{\rc}$, there are tuples $x = (x_1,\ldots,x_n)$ and $y = (y_1,\ldots,y_m)$ of elements of $G$, tuples $e = (e_1,\ldots,e_n)$ and $c = (c_1,\ldots,c_m)$ of elements of $A$, and an $\lor(K^{\p})$-definable function $\sigma$ such that
\[
a = \sigma(\Re(x_1)e_1,\ldots,\Re(x_n)e_n,\Im(y_1)c_1,\ldots,\Im(y_m)c_m)
\]
By using \cref{case1-mainthm} repeatedly, we find tuples $b = (b_1,\ldots,b_n)$ and $d = (d_1,\ldots,d_m)$ of elements of $B$ and an isomorphism
\[
\iota^{\p} : (K^{\p}(e,c)^{\rc},G^{\p},A^{\p}\ip{e,c}_A) \to (L^{\p}(b,d)^{\rc},H^{\p},B^{\p}\ip{b,d}_B)
\]
extending $\iota$ with $\iota^{\p} \in \mc{I}$.

Now let $K^{\pp} := K^{\p}(e,c)^{\rc}$ and $L^{\pp} := L^{\p}(b,d)^{\rc}$. By using \cref{case2-mainthm} repeatedly, we find $w_1,\ldots,w_n$, $z_1,\ldots,z_m \in H$, $A^{\pp} \subseteq A$, $B^{\pp} \subseteq B$, and an isomorphism
\[
\iota^{\pp}: (K^{\pp}(A^{\pp},\Re(x),\Im(y))^{\rc},G^{\p}\ip{x,y}_G,A^{\pp}) \to (L^{\pp}(B^{\pp},\Re(w),\Im(z))^{\rc},H^{\p}\ip{w,z}_H,B^{\pp})
\]
extending $\iota^{\p}$ with $\iota^{\pp} \in \mc{I}$. In particular, $A^{\p}\ip{e,c}_A \subseteq A^{\pp}$ and $B^{\p}\ip{b,d}_B \subseteq B^{\pp}$. Note that $a$ is in the domain of $\iota^{\pp}$. Thus, $\iota^{\pp} \in \mc{I}$ extends $\iota$ and has $a$ in its domain.

\begin{case-mainthm}\label{case4-mainthm} Suppose $a \in K \setminus K^{\p}(\Re(GA) \cup \Im(GA))^{\rc}$.\end{case-mainthm}

As in \cref{case2-mainthm} above, we first extend $\iota$ to an isomorphism
\[
f_{\infty} : (K^{\p}(A^{\infty})^{\rc},G^{\p},A^{\infty}) \to (L^{\p}(B^{\infty})^{\rc},H^{\p},B^{\infty})
\]
where $A^{\infty}, B^{\infty}, f_{\infty}$ are defined as in \cref{case2-mainthm}. \\

Next, we want to find $b \in L \setminus L^{\p}(\Re(HB) \cup \Im(HB))^{\rc}$ such that $b$ realizes the same cut over $L^{\p}(B^{\infty})^{\rc}$ that $a$ does over $K^{\p}(A^{\infty})^{\rc}$. We will then extend $f_{\infty}$ to an element of $\mc{I}$ that maps $a$ to $b$.

Let $\Phi_1$ be the collection of formulas of the form
\[
\neg ( \exists h = (h_1,\ldots,h_n) \in H^n \exists b = (b_1,\ldots,b_n) \in B^n [x = f(p_1(h_1b_1), \ldots, p_n(h_nb_n))])
\]
where $f$ is a $\lor(L^{\p})$-definable function from $L^n$ to $L$ and each $p_j$ is either $\Re$ or $\Im$. Thus, if $\varphi$ is a formula in $\Phi_1$ of the above form, there is an $\lor(L^{\p})$-definable function $f_{\varphi}: (HB)^n \to L$ such that
\[
f_{\varphi}(h_1b_1,\ldots,h_nb_n) = f(p_1(h_1b_1),\ldots,p_n(h_nb_n))
\]
for all $h_1,\ldots,h_n \in H$ and $b_1,\ldots,b_n \in B$. Let $\Phi_2$ be the collection of formulas
\[
\Phi_2 := \{f_{\infty}(k_1) < x < f_{\infty}(k_2) : k_1,k_2 \in K^{\p}(A^{\infty})^{\rc}, k_1 < a < k_2\}.
\]

If we have finitely many formulas $\varphi_1,\ldots,\varphi_k$ in $\Phi_1$, there are $s_1,\ldots,s_k \in \mbb{N}$ and functions $f_{\varphi_1},\ldots,f_{\varphi_k}$ such that for each $j \in \{1,\ldots,k\}$, $f_{\varphi_j}: (HB)^{s_j} \to L$ is a $\lor(L^{\p})$-definable function. We can assume that there is $m \in \mbb{N}$ such that each $f_{\varphi_j}$ is a function from $(HB)^m$ to $L$. By \cref{density-lemma}, the set $L \setminus \bigcup_{j=1}^k f_j(HB)^m)$ is dense in $L$. Therefore, given a finite subset of formulas $\Phi^{\p} \subseteq \Phi_1 \cup \Phi_2$, we can find $x \in L$ satisfying $\Phi^{\p}$. \\

By $\kappa$-saturation of $\mc{N}$, there is $b$ that satisfies all formulas in $\Phi_1 \cup \Phi_2$. This $b$ lies in $L \setminus L^{\p}(\Re(HB) \cup \Im(HB))^{\rc}$ and realizes the same cut over $L^{\p}(B^{\infty})^{\rc}$ that $a$ does over $K^{\p}(A^{\infty})^{\rc}$. Therefore, there is an ordered field isomorphism $\iota^{\p}: K^{\p}(A^{\infty},a)^{\rc} \to L^{\p}(B^{\infty},b)^{\rc}$ extending $f_{\infty}$.\\

We check that $(K^{\p}(A^{\infty},a)^{\rc},G^{\p},A^{\infty}) \in \Sub(K,G,A)$. In particular, we must show that $K^{\p}(A^{\infty},a)^{\rc}(i)$ and $\mbb{Q}(GA)$ are free over $\mbb{Q}(G^{\p}A^{\infty})$. First note that since $K^{\p}(\Re(GA))^{\rc}(i)$ is algebraically closed and $a \in K \setminus K^{\p}(\Re(GA))^{\rc}$, $a$ must be algebraically independent over $K^{\p}(\Re(GA))^{\rc}(i)$. Since $K^{\p}(GA) \subseteq K^{\p}(\Re(GA))^{\rc}(i)$, $a$ is also algebraically independent over $K^{\p}(GA)$. By construction, $K^{\p}(A^{\infty})^{\rc}$ and $\mbb{Q}(GA)$ are free over $\mbb{Q}(G^{\p}A^{\infty})$. Therefore, we may apply \cref{freeness-lemma} with $E = \emptyset$ and $X = \{a\}$ to get that $K^{\p}(A^{\infty},a)^{\rc}(i)$ and $\mbb{Q}(GA)$ are free over $\mbb{Q}(G^{\p}A^{\infty})$. We must also show that $K^{\p}(A^{\infty},a)^{\rc}$ is closed under $\lambda$. But this follows by definition of $A^{\infty}$ as in \cref{case2-mainthm}.\\

This completes the proof that $\mc{I}$ is a back-and-forth system.

\section{Proof of \cref{main-thm-statement}}\label{sec:thmA}

In this section we will give a proof of \cref{main-thm-statement}. We start by introducing the notion of a special formula and show that $T$ has quantifier-elimination up to boolean combinations of these formulas. Note that $T$ is not complete and makes no assumptions on the cardinality of $[p]G$ in a model of  $(K,G,A,(\gamma)_{\gamma \in \Gamma},(\delta)_{\delta \in \Delta})$ of $T$. Adding the requirement that $[p]G$ is finite for each prime $p$, we will establish the following stronger theorem.

\begin{thm}\label{qr-theorem}
Let $\mc{M} := (K,G,A,(\gamma)_{\gamma \in \Gamma},(\delta)_{\delta \in \Delta})$ be a model of $T$ such that $[p]G$ is finite for each prime $p$. Then every subset of $K^m$ definable in $\mc{M}$ is a boolean combination of subsets of $K^m$ defined in $\mc{M}$ by formulas of the form
\[
\exists y \exists z (V(y) \wedge P(z) \wedge \phi(x,y,z))
\]
where $\phi(x,y,z)$ is a quantifier free $\lor(K)$-formula.
\end{thm}

By \cite[Proposition 1.1]{BelZilber} every finite rank subgroup of $\mbb{S}^1(\mbb{R})$ satisfies the assumptions of \cref{qr-theorem}. The first part of Theorem A follows easily from this theorem.

\subsection{Special formulas and types} A\emph{ special $\lor(P,V,\Gamma,\Delta)$-formula} in $x$ (where $x$ is a tuple of variables) with parameters from $S$ is a formula $\psi(x)$ of the form
\[
\exists y \exists z (V(y) \wedge P(z) \wedge \theta_V^1(y) \wedge \theta_P^2(z) \wedge \phi(x,y,z))
\]
where $y$ is a tuple of variables, $z$ is a tuple of pairs of variables, $\theta^1(y)$ is an $\lom(\Delta)$-formula, $\theta^2(z)$ is an element of $\Sigma_{\orm}(\Gamma)$ (as defined in \cref{subsec:orag}), $\theta_V^1(y)$ is the $V$-restriction of $\theta^1(y)$ (recursively defined as on page 10 of \cite{Ayhan}, with $U$ replaced by $V$), $\theta_P^2(z)$ is the $P$-restriction of $\theta^2(z)$, and $\phi(x,y,z)$ is an $\lor(\Gamma,\Delta,S)$-formula. If $y = (y_1,\ldots,y_n)$ and $z = ((z_{11},z_{12}),\ldots,(z_{m1},z_{m2}))$, then $V(y)$ is an abbreviation for $V(y_1) \wedge \ldots \wedge V(y_n)$ and $P(z)$ is an abbreviation for $P(z_{11},z_{12}) \wedge \ldots \wedge P(z_{m1},z_{m2})$. By a \emph{special formula (in $x$)}, we mean a special $\lor(P,V,\Gamma,\Delta)$-formula in $x$ with parameters from $\emptyset$.\\

Now let $\mc{M} := (K,G,A,(\gamma)_{\gamma \in \Gamma},(\delta)_{\delta \in \Delta})$ be a model of $T$. Let $Y, C \subseteq K$. The \emph{special type of $Y$ over $C$}, denoted $\sptp^{\mc{M}}(Y|C)$, is the set of special formulas with parameters from $C$ satisfied by $Y$ in $\mc{M}$.\\

The following fact is Fact 1 in \cite{LouAyhan}, translated to fit our situation.

\begin{fact}\label{boolean-algebra-fact}
Let $B$ be the Boolean algebra of $T$-equivalence classes of $\lor(P,V,\Gamma,\Delta)$-formulas in the variables $x = (x_1,\ldots,x_m)$. For an $\lor(P,V,\Gamma,\Delta)$-formula $\phi(x)$, let $\phi(x)/T$ denote its $T$-equivalence class. Let $\Psi \subseteq B$ denote the set of (cosets of) special $\lor(P,V,\Gamma,\Delta)$-formulas in $x$. For an $\lor(P,V,\Gamma,\Delta)$-type $p(x)$ containing $T$, let $[p(x)] = \{\phi(x)/T : \phi(x) \in p(x)\}$. Suppose that for any $p_1,p_2 \in S_x(T)$,
\[
\text{if } [p_1(x)] \cap \Psi = [p_2(x)] \cap \Psi, \text{ then } [p_1(x)] = [p_2(x)].
\]
Then $\Psi$ generates $B$ as a Boolean algebra.
\end{fact}

Next we fix some notation that we will use in the rest of the paper. Let $\mc{L}$ be a language and let $\mc{A}$ be an $\mc{L}$-structure. Whenever $C, Y \subseteq A$, $\tp^{\mc{A}}(Y|C)$ will denote the $\mc{L}(C)$-type of $Y$. For a sublanguage $\mc{L}^{\p}$ of $\mc{L}$, $\tp_{\mc{L}^{\p}}^{\mc{A}}(Y|C)$ will denote the $\mc{L}^{\p}(C)$-type of $Y$.

Let $\mc{B}$ be another $\mc{L}$-structure and fix an injective function $f: C \to B$. We define $f(\tp_{\mc{L}^{\p}}^{\mc{A}}(Y|C))$ by
\[
f(\tp_{\mc{L}^{\p}}^{\mc{A}}(Y|C)) = \{\phi(x,f(c)) : \phi(x,z) \text{ an } \mc{L}^{\p}\text{-formula}, c \in C^{\av{y}}, \phi(x,c) \in \tp_{\mc{L}^{\p}}^{\mc{A}}(Y|C)\}.
\]
If $\mc{A}$ is an $\lor(P,V,\Gamma,\Delta)$-structure, then we define $f(\sptp^{\mc{A}}(Y|C))$ by
\[
f(\sptp^{\mc{A}}(Y|C)) = \{\phi(x,f(c)) : \phi(x,z) \text{ a special formula}, c \in C^{\av{y}}, \phi(x,c) \in \sptp^{\mc{A}}(Y|C)\}.
\]

\subsection{Quantifier elimination up to special formulas} In this section, we prove that $T$ eliminates quantifiers up to special formulas.

\begin{lem}\label{qr-main-lemma} Each $\lor(P,V,\Gamma,\Delta)$-formula $\psi(x)$ is equivalent in $T$ to a Boolean combination of special $\lor(P,V,\Gamma,\Delta)$-formulas in $x$.\end{lem}

\begin{proof} Let $\kappa > \av{\GD}$ and let
\[
\mc{M} := (K,G,A,(\gamma)_{\gamma \in \Gamma},(\delta)_{\delta \in \Delta}), \mc{N} := (L,H,B,(\gamma)_{\gamma \in \Gamma},(\delta)_{\delta \in \Delta})
\]
be $\kappa$-saturated models of $T$. Let $\alpha := (\alpha_1,\ldots,\alpha_m) \in K^m$ and $\beta := (\beta_1,\ldots,\beta_m) \in L^m$ satisfy (in $\mc{M}$ and $\mc{N}$ respectively) the same special formulas in $x$. By \cref{boolean-algebra-fact}, to prove the lemma, it suffices to show that $\tp^{\mc{M}}(\alpha) = \tp^{\mc{N}}(\beta)$. First note that the formulas expressing $[p]G$ and $[p]H$ ($p$ prime) are special formulas. Therefore, $[p]G = [p]H$. Since $\mc{M}, \mc{N} \models T$, $A$ and $B$ are regularly discrete. Therefore, for each prime $p$, $[p]A = p$ by \cref{reg-discrete-lem}. Moreover, the formula expressing "$\gamma$ is an $n$th power in $G$" is a special formula. Therefore, $\gamma$ is an $n$th power in $G$ if and only if $\gamma$ is an $n$th power in $H$ for all $\gamma \in \Gamma$. Therefore, we have a back-and-forth system $\mc{I}$ between $\mc{M}$ and $\mc{N}$ as constructed in \cref{main-thm}. To show that $\tp^{\mc{M}}(\alpha) = \tp^{\mc{N}}(\beta)$, we will find $\iota \in \mc{I}$ such that each $\alpha_j$ is in the domain of $\iota$ and $\iota(\alpha_j) = \beta_j$ for each $j \in \{1,\ldots,m\}$.\\

Throughout, let $\alpha := (\alpha_1,\ldots,\alpha_m)$ and let $\beta := (\beta_1,\ldots,\beta_m)$. Let $F = \mbb{Q}(\RGD)$. Let $\mbb{Q}(GA)(\alpha)$ have transcendence degree $r$ over $\mbb{Q}(GA)$. We may assume that $\{\alpha_1,\ldots,\alpha_r\}$ is a subset of $\{\alpha_1,\ldots,\alpha_m\}$ that is maximal with respect to being algebraically independent over $\mbb{Q}(GA)$. Thus, we have a tuple $g$ of elements of $G$, a tuple $a$ of elements of $A$, and $\lor$-definable functions $\sigma_{r+1}, \ldots, \sigma_m$ such that for each $j \in \{r+1,\ldots,m\}$, 
\[
\sigma_j(g,a,\alpha_1,\ldots,\alpha_r) = \alpha_j.
\]
By a similar argument as in Theorem 3.8 of \cite{LouAyhan}, using the fact that $\alpha$ and $\beta$ satisfy the same special formulas, $\{\beta_1,\ldots,\beta_r\}$ is algebraically independent over $\mbb{Q}(HB)$.\\

We first define $A^{\infty} \subseteq A$ in a similar way as in \cref{case2-mainthm} of this theorem. That is, let $A^{(1)} = \Delta$ and for $j \geq 1$, define $K_j = F(A^{(j)})^{\rc}$ and
\[
A^{(j+1)} = \lambda(K_j(\alpha, g, a)^{\rc}).
\]
Let $A^{\infty} = \bigcup_{j \geq 1} A^{(j)}$. Note that since $\kappa > \av{\GD}$, we have $\av{A^{\infty}} < \kappa$. Let $\av{A^{\infty}} = \rho$. Consider $\sptp^{\mc{M}}(g,A^{\infty}|\alpha)$ as a set of $\lor(P,V,\Gamma,\Delta)$-formulas in the variables $(x_{\eta} : \eta < \rho)$. We will show that $\sptp^{\mc{M}}(g,A^{\infty}|\alpha)$ is finitely satisfiable in $\mc{N}$ when each $\alpha_j$ is replaced by the corresponding $\beta_j$. Let $\mu: \{\alpha_1,\ldots,\alpha_m\} \to \{\beta_1,\ldots,\beta_m\}$ be the function defined by $\mu(\alpha_i) = \beta_i$ for each $i$. Suppose $\phi_1(c,g,\alpha),\ldots,\phi_n(c,g,\alpha)$ are formulas in $\sptp^{\mc{M}}(g,A^{\infty}|\alpha)$, where $\phi_1,\ldots,\phi_n$ are $\lor(P,V,\Gamma,\Delta)$-formulas and $c$ is a tuple of elements of $A^{\infty}$. Then
\[
\phi(\alpha) := \exists y \in V \exists z \in P (\phi_1(y,z,\alpha) \wedge \ldots \wedge \phi_n(y,z,\alpha))
\]
is equivalent to a special formula $\psi(\alpha)$ in $\sptp^{\mc{M}}(\alpha)$. By our assumption that $\sptp^{\mc{M}}(\alpha) = \sptp^{\mc{N}}(\beta)$, we have $\psi(\beta) \in \sptp^{\mc{N}}(\beta)$. By $\kappa$-saturation of $\mc{N}$, there is a subset $B^{\infty}$ of elements of $B$ and tuple $h$ of elements of $H$ such that $\mu(\sptp^{\mc{M}}(g,A^{\infty}|\alpha)) \subseteq \sptp^{\mc{N}}(h,B^{\infty}|\beta)$. In particular,
\[
\mu(\tp_{\lor}^{\mc{M}}(g,A^{\infty}|\alpha)) = \tp_{\lor}^{\mc{N}}(h,B^{\infty}|\beta). \tag{$*$}
\]
Let $\mc{G}$ denote the $\lorm$-structure with universe $G$, with the orientation and multiplication on $G$ interpreted as in \cref{subsec:orag}. Let $\mc{H}$ denote the $\lorm$-structure with universe $H$, again with orientation and multiplication interpreted as in \cref{subsec:orag}. Since $\mu(\sptp^{\mc{M}}(g,A^{\infty}|\alpha)) \subseteq \sptp^{\mc{N}}(h,B^{\infty}|\beta)$, we also have
\[
\tp^{\mc{G}}(g) = \tp^{\mc{H}}(h). \tag{$**$}
\]

Note that each equation $\sigma_j(c,g^{\p},\alpha_1,\ldots,\alpha_r) = \alpha_j$ (for $j \in \{r+1,\ldots,m\}$ and tuples $c$ of elements from $A^{\infty}$ and $g^{\p}$ of elements from $G$) corresponds to a special formula in $\sptp^{\mc{M}}(g,A^{\infty}/\alpha)$. Therefore, for $j \in \{r+1,\ldots,m\}$, we also have
\[
\sigma_j(d,h^{\p},\beta_1,\ldots,\beta_r) = \beta_j
\]
for some tuples $d$ of elements from $A^{\infty}$ and $h^{\p}$ of elements from $h$.\\

Let
\begin{align*}
K^{\p} &= F(\alpha,g,A^{\infty})^{\rc}, \ \ G^{\p} = \Gamma\ip{g}_G, \ \ A^{\p} = A^{\infty},\\
L^{\p} &= F(\beta,h,B^{\infty})^{\rc}, \ \ H^{\p} = \Gamma\ip{h}_H, \ \ B^{\p} = B^{\infty}.
\end{align*}
By $(*)$, we have an ordered field isomorphism $\iota: K^{\p} \to L^{\p}$ which takes $g$ to $h$, $A^{\infty}$ to $B^{\infty}$, and $\alpha$ to $\beta$. We claim that $\iota \in \mc{I}$. By construction, $\iota(A^{\infty}) = B^{\infty}$. We now want to show that $\iota(G^{\p}) = H^{\p}$. To do this, it suffices to show that for all $\gamma \in \Gamma$, $p_1,\ldots,p_k \in \mbb{Z}$, and $n > 0$,
\[
\gamma g_1^{p_1} \ldots g_k^{p_k} \in G^{[n]}\text{ if and only if } \gamma h_1^{p_1} \ldots h_k^{p_k} \in H^{[n]}.
\]
But this follows from $(**)$.\\

We now show that $(K^{\p},G^{\p},A^{\p}) \in \Sub(K,G,A)$ and $(L^{\p},H^{\p},B^{\p}) \in \Sub(L,H,B)$. In particular, we must show that $K^{\p}(i)$ and $\mbb{Q}(GA)$ are free over $\mbb{Q}(G^{\p}A^{\p})$. Let $k = F(\Re(g),A^{\infty})$ and let $X = \{\alpha_1,\ldots,\alpha_r\}$. Note that $k \subseteq \mbb{Q}(\Re(G^{\p}A^{\p}))$. By our choice of $X$, $X$ is algebraically independent over $\mbb{Q}(GA)$. Therefore, we may apply \cref{freeness-cor} to get that $K^{\p}(i)$ and $\mbb{Q}(GA)$ are free over $\mbb{Q}(G^{\p}A^{\p})$. The fact that $K^{\p}$ is closed under $\lambda$ follows by definition of $K^{\p}$. Next we must check that $L^{\p}(i)$ and $\mbb{Q}(HB)$ are free over $\mbb{Q}(H^{\p}B^{\p})$. Let $Y = \{\beta_1,\ldots,\beta_r\}$. As stated previously, $Y$ is algebraically independent over $\mbb{Q}(HB)$. Therefore, a similar proof as before shows that $L^{\p}(i)$ and $\mbb{Q}(HB)$ are free over $\mbb{Q}(H^{\p}B^{\p})$. Lastly, we must check that $L^{\p}$ is closed under $\lambda$. Since $\iota(A^{\infty}) = B^{\infty}$, we have $y = \sigma(\beta_1,\ldots,\beta_r,h,\iota(c))$ for some $\lor$-definable function $\sigma$ and tuple $c$ of elements from $A^{\infty}$. Consider $x := \lambda(\sigma(\alpha,g,c))$. Then
\[
x \leq \sigma(\alpha_1,\ldots,\alpha_r,g,c) < \bdelta x.
\]
Since $\iota$ is an ordered field isomorphism which takes $g$ to $h$ and takes $\alpha_i$ to $\beta_i$ for $i \in \{1,\ldots,r\}$, we have
\[
\iota(x) \leq \sigma(\beta_1,\ldots,\beta_r,h,\iota(c)) < \bdelta\iota(x).
\]
Therefore, $\iota(x) = \lambda(y)$. Since $\iota(x) \in B^{\infty}$, we have $\lambda(y) \in B^{\infty}$.\\

Therefore, $\iota \in \mc{I}$, $\alpha_1,\ldots,\alpha_m$ are in the domain of $\mc{I}$, and $\iota(\alpha_j) = \beta_j$ for each $j$. This finishes the proof of the lemma. \end{proof}

\subsection{Quantifier elimination lemmas}\label{subsec:qe-lemmas}

In this section, we prove some further quantifier elimination lemmas necessary for the proof of Theorem \ref{qr-theorem}. To be precise, we establish a quantifier elimination results for regularly dense oriented abelian groups, and recall another quantifier elimination result for regularly discrete abelian groups. The main tool for the first result is \cref{claim-812} (see also Section 8 of \cite{Ayhan}).\\

In the following two lemmas, let $\mc{L}_1$ be the language of oriented monoids together with constants from an infinite multiplicative subgroup $\Gamma \subseteq \mbb{S}^1$. That is, $\mc{L}_1 = \{ \mc{O}, \cdot, 1,(\gamma)_{\gamma \in \Gamma}\}$, where $\mc{O}$ is a ternary predicate. Fix a function $e$ from the set of prime numbers to $\mbb{N}$. Let $\mc{L}_2 = \mc{L}_1 \cup \{E_n : n > 0\}$, where each $E_n$ is a unary predicate with defining axiom
\[
\sigma_n := \forall z (E_n(z) \leftrightarrow \exists y (z = y^n)).
\]

\begin{lem}\label{qe-lem-1}
Let $T_1(e)$ be the theory of regularly dense oriented abelian groups $G$ containing $\Gamma$ as a subgroup such that $[p]G = p^{e(p)}$ for all primes $p$ and $G_{\tor} = \Gamma_{\tor}$. Let $T_2 = \{\sigma_n : n > 0\}$. Then the $\mc{L}_2$-theory $\Sigma(e) := T_1(e) \cup T_2$ admits quantifier elimination.
\end{lem}

\begin{proof}
Let $G,H$ be $\mc{L}_2$-structures such that $G,H \models \Sigma(e)$ and $H$ is $\kappa^+$-saturated for some $\kappa \geq \av{G}$. Let $G^{\p}$ be a proper $\mc{L}_2$-substructure of $G$ and let $f: G^{\p} \to H$ be an embedding. Let $g \in G \setminus G^{\p}$, and let $H^{\p}$ be the substructure of $H$ with $H^{\p} = f(G^{\p})$. We will find $h \in H \setminus H^{\p}$ such that $f$ extends to an $\mc{L}_2$-isomorphism $f^{\p}: G^{\p}\ip{g}_G \to H^{\p}\ip{h}_H$.\\

Note that by definition of $\Sigma(e)$, we have $[p]G = [p]H$ for all primes $p$. Since $G^{\p} \subseteq G$, we also have $\Gamma_{\tor} \subseteq G^{\p}_{\tor} \subseteq G_{\tor} = \Gamma_{\tor}$. Therefore, $G^{\p}_{\tor} = G_{\tor}$. Similarly, $H^{\p}_{\tor} = H_{\tor}$. If $g^{\p} \in G^{\p}$ has an $n$th root in $G$, then $G \models E_n(g^{\p})$. Since $G^{\p} \subseteq G$, we also have $G^{\p} \models E_n(g^{\p})$, so $G^{\p}$ is pure in $G$. Likewise, $H^{\p}$ is pure in $H$.\\

By \cref{claim-812}, there is $h \in H$ such that there is an oriented group isomorphism $f^{\p}: G^{\p}\ip{g}_G \to H^{\p}\ip{h}_H$ extending $f$ and taking $g$ to $h$. Note that since $f^{\p}$ extends $f$, we have $f^{\p}(\gamma) = \gamma$ for all $\gamma \in \Gamma$. Therefore, $f^{\p}$ is an $\mc{L}_1$-isomorphism between $G^{\p}\ip{g}_G$ and $H^{\p}\ip{h}_H$. We now check that $f^{\p}$ is an $\mc{L}_2$-isomorphism.\\

Suppose $G^{\p}\ip{g}_G \models E_n(b)$ for some $b \in G^{\p}\ip{g}_G$. By definition of $E_n$, $b = y^n$ for some $y \in G^{\p}\ip{g}_G$. Since $f^{\p}$ is an oriented group isomorphism, we have $f^{\p}(b) = (f^{\p}(y))^n$. So $H^{\p}\ip{h}_H \models E_n(f^{\p}(b))$. Conversely, if $H^{\p}\ip{h}_H \models E_n(f^{\p}(b))$ for some $b \in G^{\p}\ip{g}_G$, then $G^{\p}\ip{g}_G \models E_n(b)$. Therefore, $f^{\p}$ is an $\mc{L}_2$-isomorphism between $G^{\p}\ip{g}_G$ and $H^{\p}\ip{h}_H$.\\

Since we have found an $\mc{L}_2$-isomorphism $f^{\p}$ properly extending $f$, $\Sigma$ has quantifier elimination.
\end{proof}

\begin{lem}\label{qe-lem-2}
Let $\Phi$ be a set of $\mc{L}_2$-sentences that axiomatizes the class of abelian groups. Every atomic $\mc{L}_2$-formula $\varphi(x_1,\dots,x_n)$ is equivalent in $\Phi$ to a formula with one of the following forms:
\[
\gamma^k x_1^{k_1} \ldots x_n^{k_n} = 1
\]
\[
\mc{O}(\gamma_1^kx_1^{k_1} \ldots x_n^{k_n}, \gamma_2^l y_1^{l_1} \ldots y_m^{l_m}, \gamma_3^p z_1^{i_1} \ldots z_p^{i_p})
\]
\[
E_d(\gamma^k x_1^{k_1} \ldots x_n^{k_n})
\]
where $k_1,\ldots,k_n$, $l_1,\ldots,l_m$, $i_1,\ldots,i_p \in \mbb{Z}$, $k,l,i$ are tuples of elements of $\mbb{Z}$, $d$ is a positive integer, and $\gamma, \gamma_1,\gamma_2,\gamma_3$ are tuples of elements from $\Gamma$.
\end{lem}

\begin{proof}
Since any $\mc{L}_2$-structure which models $\Phi$ is an abelian group, one can show by induction on terms that in any model of $\Phi$, every $\mc{L}_2$-term is equal to a term of the form $\gamma^k x_1^{k_1} \ldots x_n^{k_n}$. Thus, it is clear that every $\mc{L}_2$-atomic formula must be equivalent in $\Phi$ to a formula with one of the above forms.
\end{proof}

We next recall some results for regularly discrete abelian groups. In the following two lemmas, let $\mc{L}_3 = \{\cdot,<,1,\bdelta\} \cup \{D_n : n > 0\}$, where each $D_n$ is a unary predicate. For $n > 0$, let
\[
\tau_n := \forall x (D_n(x) \leftrightarrow \exists y (x = y^n)).
\]

\begin{lem}\label{qe-lem-3}
Let $T_3$ be the $\mc{L}_3$-theory of regularly discrete abelian groups $A$ with $\bdelta$ the smallest element larger than 1, together with the set of sentences $\{\tau_n : n > 0\}$. Then $T_3$ admits quantifier elimination.
\end{lem}

\begin{proof}
By \cref{reg-discrete-lem}, for each $n \geq 1$ and each $a \in A$, there is $i \in \{1,\ldots,n\}$ such that $a \bdelta^i \in A^{[n]}$. Therefore, the theory $T_3$ includes the sentence
\[
\forall a (D_n(a\bdelta) \vee \ldots \vee D_n(a\bdelta^n))
\]
for each $n \geq 1$. Since the theory of $\mbb{Z}$-groups admits quantifier elimination, $T_3$ admits quantifier elimination.
\end{proof}

From this, the following lemma follows easily by a similar proof as \cref{qe-lem-2}.

\begin{lem}\label{qe-lem-4}
Let $\Phi^{\p}$ be a set of $\mc{L}_3$-sentences that axiomatizes the class of abelian groups. Every $\mc{L}_3$-atomic formula $\varphi(x_1,\dots,x_n)$ is equivalent in $\Phi^{\p}$ to a formula with one of the following forms:
\[
\delta^k x_1^{k_1} \ldots x_n^{k_n} = 1
\]
\[
\delta^k x_1^{k_1} \ldots x_n^{k_n} < 1
\]
\[
D_d(\delta^k x_1^{k_1} \ldots x_n^{k_n})
\]
where $k_1,\ldots,k_n \in \mbb{Z}$, $k$ is a tuple of elements of $\mbb{Z}$, $d > 0$, and $\delta$ is a tuple of elements from $\bdelta^{\mbb{Z}}$.
\end{lem}

\subsection{Proof of Theorem \ref{qr-theorem}}\label{subsec:qr-thm} By \cref{qr-main-lemma}, it suffices to show that subsets of $K^m$ defined by special formulas have the desired form.
Let
\[
\psi(x) := \exists y \exists z (V(y) \wedge P(z) \wedge \theta_V^1(y) \wedge \theta_P^2(z) \wedge \phi(x,y,z))
\]
be a special formula where $y = (y_1,\ldots,y_n)$ is a tuple of variables and $$z = ((z_{11},z_{12}),\ldots,(z_{j1},z_{j2}))$$ is a tuple of pairs of variables.\\

By \cref{qe-lem-3} and \cref{qe-lem-4}, the set $\{a \in A^n: A \models \theta_V^1(a)\}$ is a Boolean combination of subsets of $A^n$, each of which has one of the following forms:
\[
\{a \in A^n : \delta^k a_1^{k_1} \ldots a_n^{k_n} = 1\} \text{ or }
\]
\[
\{a \in A^n : \delta^k a_1^{k_1} \ldots a_n^{k_n} < 1\} \text{ or }
\]
\[
\{a \in A^n : \delta^k a_1^{k_1} \ldots a_n^{k_n} \in A^{[d]}\}
\]
where $k_1,\ldots,k_n \in \mbb{Z}$, $\delta \in \Delta$, $k$ is a tuple of elements from $\mbb{Z}$, and $d$ is a positive integer. Since $\mc{M} \models T$, $[d]A$ is finite. Therefore, a set of the form
\[
A^n \setminus \{a \in A^n : \delta^k a_1^{k_1} \ldots a_n^{k_n} \in A^{[d]}\}
\]
is equal to a finite union of sets of the form
\[
\{a \in A^n : \delta^k a_1^{k_1} \ldots a_n^{k_n} \in a^{\p}A^{[d]}\}
\]
where $a^{\p} \in A$. Therefore, $\psi(x)$ is equivalent in $\mc{M}$ to a formula $\psi^{\p}(x)$ with
\[
\psi^{\p}(x) := \exists y^{\p} \exists z (V(y^{\p}) \wedge P(z) \wedge \theta_P^2(z) \wedge \phi^{\p}(x,y^{\p},z))
\]
where $y^{\p}$ is a tuple of variables (extending $y$) and $\phi^{\p}$ is an $\lor(A)$-formula.\\

Now consider the subgroup $G \subseteq \mbb{S}^1(K)$. Since we assume $[p]G$ is finite for each prime $p$, there is a function $e$ from the set of prime numbers to $\mbb{N}$ such that $[p]G = p^{e(p)}$ for each prime $p$. Let $\mc{L}_1$ and $\mc{L}_2$ be the languages defined before \cref{qe-lem-1}, and let $\Sigma(e)$ be the set of $\mc{L}_2$-sentences defined in \cref{qe-lem-1}. We can make $G$ into an $\mc{L}_2$-structure $\mc{G}$ such that $\mc{G} \models \Sigma(e)$ by taking $\mc{O}$ to be the orientation on $G$ inherited from $\mbb{S}^1(K)$. Now note that there is an $\mc{L}_2$-formula $\theta(z)$ such that for all $g \in G$, $\mc{M} \models \theta_P^2(g)$ if and only if $\mc{G} \models \theta(g)$. By \cref{qe-lem-1} and \cref{qe-lem-2}, $\theta$ is equivalent in $\mc{G}$ to a Boolean combination of subsets of $G^j$, each of which has one of the following forms:
\[
\gamma^k x_1^{k_1} \ldots x_n^{k_n} = 1 \text{ or }
\]
\[
\mc{O}(\gamma_1^k x_1^{k_1} \ldots x_n^{k_n}, \gamma_2^k y_1^{k_1} \ldots y_m^{k_m}, \gamma_3^k z_1^{i_1} \ldots z_p^{i_p}) \text{ or } 
\]
\[
E_d(\gamma^k x_1^{k_1} \ldots x_n^{k_n})
\]
where $k_1,\ldots,k_n,l_1,\ldots,l_m,i_1,\ldots,i_p \in \mbb{Z}$, $k,l,i$ are tuples of elements of $\mbb{Z}$, $d$ is a positive integer, and $\gamma,\gamma_1,\gamma_2,\gamma_3$ are tuples of elements from $\Gamma$. Since we assume $[p]G$ is finite for each prime $p$, a set of the form
\[
G^j \setminus \{g \in G^j : \gamma^k g_1^{k_1} \ldots g_n^{k_n} \in G^{[d]}\}
\]
is equal to a finite union of sets of the form
\[
\{g \in G^j : \gamma^k g_1^{k_1} \ldots g_n^{k_n} \in g^{\p}G^{[d]}\}
\]
where $g^{\p} \in G$. Therefore, there is an existential $\mc{L}_1(G)$-formula $\theta^{\p}(z)$ such that for all $g \in G$, $\mc{G} \models \theta(g) \leftrightarrow \theta^{\p}(g)$. Therefore, $\psi^{\p}(x)$ is equivalent in $\mc{M}$ to a formula $\psi^{\pp}(x)$ with
\[
\psi^{\pp}(x) := \exists y^{\p} \exists z^{\p} (V(y^{\p}) \wedge P(z^{\p}) \wedge \phi^{\pp}(x,y^{\p},z^{\p}))
\]
where $z^{\p}$ is a tuple of pairs of variables extending $z$ and $\phi^{\pp}$ is an $\lor(A,G)$-formula.\\

Note that $\phi^{\pp}$ may not be quantifier free; however, by quantifier elimination for real closed fields, we can find a quantifier free $\lor(K)$-formula $\chi(x,y^{\p},z^{\p})$ such that
\[
\mc{M} \models \forall x (\exists y^{\p} \exists z^{\p} (V(y^{\p} \wedge P(z^{\p}) \wedge \phi^{\pp}(x,y^{\p},z^{\p})) \leftrightarrow \exists y^{\p} \exists z^{\p} (V(y^{\p}) \wedge P(z^{\p}) \wedge \chi(x,y^{\p},z^{\p}))).
\]
Therefore, $\psi$ is equivalent in $\mc{M}$ to a formula of the desired form.

\section{Definable open sets}\label{subsec:opendef}

In this section, we prove the statement about open definable sets in \cref{main-thm-statement}. Let $\Gamma$ be a finite rank subgroup of $\mbb{S}^1$ which is dense in $\mbb{S}^1$ and let $\Delta$ be as before. Let $\mc{L}^{*} = \lor(P,V)$ and let $\mc{L} = \lor(V)$. Let $\mc{R}^* = (\bar{\mbb{R}},\Gamma,\Delta)$ and let $\mc{R} = (\bar{\mbb{R}},\Delta)$. Using Corollary 3.1 from Boxall and Hieronymi \cite{BH}, we will show that every open set definable in $\mc{R}^*$ is already definable in $\mc{R}$. Throughout this section, ``open" will mean open in the usual order topology.\\

Let $\mc{M}^*$ be a $\kappa$-saturated, strongly $\kappa$-homogeneous elementary extension of $\mc{R}^*$ (where $\kappa = \av{\mbb{R}}^+$) with $\mc{M}^* = (M,G,A)$. Let $\mc{M}$ be the reduct of $\mc{M}^*$ to $\mc{L}$, so that $\mc{M} = (M,A)$. Let $\mc{N}$ be the reduct of $\mc{M}$ to $\lor$. Let $C$ be a countable subset of $M$. Note $A$ has a smallest element larger than 1 which we again denote by $\bdelta$. Moreover, since $\GD$ has the Mann property, $GA$ also has the Mann property.\\

We need to fix some notation. 
 For $S \subseteq M$, denote $\lor$-definable closure of $S$ in $\mc{M}$ by  $\dcl_{\lor}^{\mc{M}}(S)$. 
 By Pillay and Steinhorn \cite{pillay-steinhorn-I}, the closure operator $\dcl_{\lor}^{\mc{M}}$ has the exchange property. That is, for $a,b \in M$ and $S \subseteq M$, if $a \in \dcl_{\lor}^{\mc{M}}(S \cup \{b\})$ and $a \notin \dcl_{\lor}^{\mc{M}}(S)$, then $b \in \dcl_{\lor}^{\mc{M}}(S \cup \{a\})$.

\begin{thm} Every open set definable in $(\bar{\mbb{R}},\Gamma,\Delta)$ is definable in $(\bar{\mbb{R}},\Delta)$. \end{thm}

\begin{proof}
For $n \geq 1$, let
\[
D_n = \{(a_1,\ldots,a_n) \in M^n : \{a_1,\ldots,a_n\} \text{ is } \dcl_{\lor}^{\mc{M}}\text{-independent over } G \cup A \cup C\}.
\]
Since the topology on $M$ is the order topology, Assumption (I) in \cite{BH} is satisfied. To prove the theorem, we will apply Corollary 3.1 in \cite{BH}. To do this, we must also check that for all $n \geq 1$,
\begin{enumerate}
\item $D_n$ is dense in $M^n$;
\item for every $a \in D_n$ and every open set $U \subseteq M^n$, if $\tp^{\mc{M}}(a|C)$ is realized in $U$, then $\tp^{\mc{M}}(a|C)$ is realized in $U \cap D_n$;
\item for every $a,b \in D_n$, if $b$ realizes $\tp^{\mc{M}}(a|C)$, then $b$ realizes $\tp^{\mc{M}^*}(a|C)$.
\end{enumerate}

(1): Let $U := (c_{11},c_{12}) \times \ldots \times (c_{n1},c_{n2})$ be a basic open set in $M^n$. We want to show that there exists $a \in M^n$ with $a \in U \cap D_n$.\\

Let $S = G \cup A \cup C$. We first find $a_1 \in M$ such that $a_1 \in (c_{11},c_{12}) \setminus \dcl_{\lor}^{\mc{M}}(S)$. We use $\kappa$-saturation of $\mc{M}$ to show that there is $x \in M$ such that $x \in (c_{11},c_{12}) \setminus \dcl_{\lor}^{\mc{M}}(S)$.\\

Let $f_1,\ldots,f_l$ be $\lor(C)$-definable functions from $M^{2n}$ to $M$. For each $i$, let $X_i := \{f_i(ga) : ga \in (GA)^n \}$. By \cref{density-lemma}, the set $M \setminus \bigcup_{i=1}^l X_i$ is dense in $M$. In particular, there is $y \in (c_{11},c_{12})$ such that for all $i$ and all tuples $ga \in (GA)^n$, $y \neq f_i(ga)$. By $\kappa$-saturation of $\mc{M}$, there is $a_1 \in M$ with $a_1 \in (c_{11},c_{12}) \setminus \dcl_{\lor}^{\mc{M}}(S)$.\\

We now want to show that  there is $a_2 \in M$ such that $(a_1,a_2) \in (c_{11},c_{12}) \times (c_{21},c_{22})$ and $(a_1,a_2)$ is $\dcl_{\lor}^{\mc{M}}$-independent over $S$. By the exchange property of $\dcl_{\lor}^{\mc{M}}$, it is enough to find $a_2 \in (c_{21},c_{22}) \setminus \dcl_{\lor}^{\mc{M}}(S \cup \{a_1\})$. But such an $a_2$ exists by a similar proof as in the previous paragraph. Continuing in this way, we can find $a_1,\ldots,a_n \in M$ such that $(a_1,\ldots,a_n) \in D_n \cap U$.\\

(2): Let $a \in D_n$ and $U \subseteq M^n$, and suppose $\tp^{\mc{M}}(a|C)$ is realized in $U$. Let $d$ be a realization of $\tp^{\mc{M}}(a|C)$ in $U$. We will show that every finite subset of $\tp^{\mc{M}}(a|C)$ is realized in $U \cap D_n$. By $\kappa$-saturation of $\mc{M}$, this suffices to prove that $\tp^{\mc{M}}(a|C)$ is satisfied in $U \cap D_n$.\\

Let $\varphi_1(x),\ldots,\varphi_n(x) \in \tp^{\mc{M}}(a|C)$ (so $\varphi_1,\ldots,\varphi_n$ are $\mc{L}(C)$-formulas). By Corollary 4.1.7 in Tychonievich \cite{tycho}, for each $i$, $\varphi_i(x)$ is equivalent to a formula of the form $\exists y \in V^{m_i} \theta_i(y,x)$. where $\theta_i$ is an $\lor(C)$-formula. We claim that for each $i$, the set
\[
A_i := \{x \in M^n : \mc{M} \models \exists y \in V^{m_i} \theta_i(y,x)\}
\]
has interior, and its interior contains $d$.\\

First note that since $\tp^{\mc{M}}(a|C) = \tp^{\mc{M}}(d|C)$, we have $d \in A_i$. Fix $\alpha \in A^{m_i}$ such that $\mc{M} \models \theta_i(\alpha,d)$, and let
$B_i(\alpha) = \{x \in M^n : \mc{M} \models \theta_i(\alpha,x)\}$.
Since $\mc{N}$ is o-minimal, let $\mc{D}$ be a decomposition of $M^n$ into cells which partitions $B_i(\alpha)$, and let $X$ be the cell in this decomposition which contains $d$. Using the notation of \cite{tame-topology}, let $X$ be an $(i_1,\ldots,i_n)$-cell. We will show that $X$ is an open cell. Suppose not; then for some $j \in \{1,\ldots,n\}$, we must have $i_j = 0$. It can be shown that since $B_i(\alpha)$ is definable by an $\lor(A \cup C)$-formula, $X$ is also definable by an $\lor(A \cup C)$-formula. Since $i_j = 0$, there is an $\lor(C)$-definable function $f$ and parameters $\beta \in A^l$ such that $f(\beta,d_1,\ldots,d_{j-1}) = d_j$. Consider the $\mc{L}(C)$-formula
\[
\exists y_1 \in V \ldots \exists y_l \in V \ f(y_1,\ldots,y_l,x_1,\ldots,x_{j-1}) = x_j.
\]
This formula is in $\tp^{\mc{M}}(d|C)$, hence it is also in $\tp^{\mc{M}}(a|C)$. But then there are parameters $\beta^{\p} \in A^l$ such that $f(\beta^{\p},a_1,\ldots,a_{j-1}) = a_j$. Therefore, $a_j \in \dcl_{\lor}^{\mc{M}}(S)$, contradicting our assumption that $a\in D_n$. Therefore, $X$ is an open cell containing $d$, so $d \in \intr(A_i)$.\\

Now let $V = U \cap \left( \bigcap_{i=1}^n \intr(A_i) \right)$. Since $V$ is a finite intersection of open sets, $V$ is open. Moreover, since $d \in \intr(A_i)$ for each $i$ and since $d \in U$, $V$ is nonempty. Since $D_n$ is dense in $M^n$, there is $b \in M$ such that $b \in V \cap D_n$. Since $b \in \intr(A_i) \subseteq A_i$ for each $i$, we have $\mc{M} \models \varphi_i(b)$ for each $i$. Therefore, $\tp^{\mc{M}}(a|C)$ is finitely satisfiable in $U \cap D_n$. By $\kappa$-saturation of $\mc{M}$, it is satisfiable in $U \cap D_n$.\\

(3): Let $a,b \in D_n$ and suppose that $b$ satisfies $\tp^{\mc{M}}(a|C)$. We want to show that $b$ satisfies $\tp^{\mc{M}^*}(a|C)$. Note that since $\mc{M}^* \succeq (\bar{\mbb{R}},\Gamma,\Delta)$, we have $\mc{M}^* \models T$. Let $\mc{I}$ denote the back-and-forth system constructed in \cref{main-thm} for $\mc{M}^*$ and $\mc{M}^*$. To show that $b$ satisfies $\tp^{\mc{M}^*}(a|C)$, it suffices to show that there is $\iota \in \mc{I}$ such that $\iota$ fixes $C$ pointwise and $\iota(a_i) = b_i$ for all $i$. Let $\mu: \{a_1,\ldots,a_n\} \to \{b_1,\ldots,b_n\}$ be the function defined by $\mu(a_i) = b_i$ for each $i$. \\

Note that since $a,b \in D_n$, $\{a_1,\ldots,a_n\}$ and $\{b_1,\ldots,b_n\}$ are both algebraically independent over $\mbb{Q}(GA)$. We first construct $A^{\infty} \subseteq A$ as in \cref{qr-main-lemma}. Let $p = \tp_{\lor}^{\mc{M}^*}(A^{\infty}|a)$. By our assumption that $\tp^{\mc{M}}(a) = \tp^{\mc{M}}(b)$, $p$ is finitely satisfiable by elements of $A$ in $\mc{M}^*$. Thus, by $\kappa$-saturation of $\mc{M}^*$, there is a subset $B^{\infty}$ of elements of $A$ such that $\mu(\tp_{\lor}^{\mc{M}^*}(A^{\infty}|a)) \subseteq \tp_{\lor}^{\mc{M}^*}(B^{\infty}|b)$.\\

Let $F = \mbb{Q}(\RGD)$ and let
\begin{align*}
K^{\p} = F(a,A^{\infty})^{\rc}, \ \ G^{\p} = \ip{\Gamma}_G, \ \ A^{\p} = A^{\infty}, \\
L^{\p} = F(b,B^{\infty})^{\rc}, \ \ H^{\p} = \ip{\Gamma}_G, \ \ B^{\p} = B^{\infty}.
\end{align*}
By construction, $\mu(\tp_{\lor}^{\mc{M}^*}(A^{\infty}|a)) \subseteq \tp_{\lor}^{\mc{M}^*}(B^{\infty}|b)$. So we have an ordered field isomorphism $\iota: K^{\p} \to L^{\p}$ which takes $A^{\infty}$ to $B^{\infty}$ and $a$ to $b$. The fact that $\iota \in \mc{I}$ follows as in \cref{qr-main-lemma}.\\

To finish the proof, we must show that if $U \subseteq \mbb{R}^m$ is an open definable set in $\mc{R}^*$, then $U$ is definable in $\mc{R}$. Let $U \subseteq \mbb{R}^m$ be an open definable set in $\mc{R}^*$. Then $U$ is definable with finitely many parameters from $\mbb{R}$, say $\{\alpha_1,\ldots,\alpha_n\}$. Let $\alpha = (\alpha_1,\ldots,\alpha_n)$ and let $\phi(y,x)$ be an $\mc{L}^{\p}$-formula defining $U$, so that
\[
U = \{x \in \mbb{R}^m : \mc{R}^* \models \phi(\alpha,x)\}.
\]
Since $\mc{R}^* \preceq \mc{M}^*$, the set $V := \{x \in M^m : \mc{M}^* \models \phi(\alpha,x)\}$ is open and definable (in $\mc{M}^*$) over the set $\{\alpha_1,\ldots,\alpha_n\}$. By Corollary 3.1 in \cite{BH}, $V$ is definable in $\mc{M}$ over $\{\alpha_1,\ldots,\alpha_n\}$. Let $\psi$ be an $L$-formula such that $V = \{x \in M^m : \mc{M} \models \psi(\alpha,x)\}$. Now consider the definable set $U^{\p} := \{x \in \mbb{R}^m : \mc{R} \models \psi(\alpha,x)\}$. Since $\mc{R} \preceq \mc{M}$ and $\mc{R}^* \preceq \mc{M}^*$, we have $U^{\p} = U$. Therefore, $U$ is definable in $\mc{R}$.
\end{proof}

\section*{Acknowledgements}

The author would like to thank Philipp Hieronymi for his help and guidance in writing this paper. The author would also like to thank the referee for their many helpful comments. This work was supported in part by a gift to the Mathematics Department at the University of Illinois from Gene H. Golub and Hieronymi's NSF grant DMS-1300402 and UIUC Campus Research Board award 14194.

\end{document}